\documentclass[11pt,twoside]{article}
\usepackage[utf8]{inputenc}
\usepackage{amsmath,amssymb,latexsym}
\usepackage{amsthm}
\usepackage{a4wide}
\usepackage{fourier}
\usepackage{array}
\usepackage{stackengine}
\usepackage{faktor}
\usepackage{fancyhdr}
\usepackage{graphicx}
\usepackage[svgnames]{xcolor}
\usepackage[labelfont=bf,width=0.8\textwidth,justification=centering]{caption}
\usepackage{tikz}
\usepackage{subfigure}
\usepackage[colorlinks=true,linkcolor=Indigo,citecolor=Indigo]{hyperref
 }
  \makeatletter
 \newenvironment{sqcases}{%
  \matrix@check\sqcases\env@sqcases
}{%
  \endarray\right.%
}
\def\env@sqcases{%
  \let\@ifnextchar\new@ifnextchar
  \left\lbrack
  \def\arraystretch{1.2}%
  \array{@{}l@{\quad}l@{}}%
}
\makeatother
\tikzset{every picture/.style={line width=0.75pt}}
 \usetikzlibrary{patterns}
\newtheorem{theorem}{Theorem}[section]
\newtheorem{proposition}[theorem]{Proposition}
\newtheorem{definition}[theorem]{Definition}
\newtheorem{lemma}[theorem]{Lemma}
\theoremstyle{definition}
\newtheorem{observation}[theorem]{Observation}
\newtheorem{rmk}[theorem]{Remark}
\numberwithin{equation}{section}
\numberwithin{figure}{section}

\title{Collision trajectories and regularisation of two-body problem on $S^2$}
\author{Alessandro Arsie \ \& \ Nataliya A. Balabanova }
\date{August 2022}

\def\bq{\mathbf{q}}

\def\bp{\mathbf{p}}

\numberwithin{equation}{section}
\numberwithin{figure}{section}

\begin{document}

  \maketitle 
  \begin{abstract}
     In this paper, we investigate collision orbits of two identical bodies placed on the surface of a two-dimensional sphere and interacting via an attracting potential of the form $V(q)=-\cot(q)$, where $q$ is the angle formed by the position vectors of the two bodies. We describe the $\omega$-limit set of the variables in the symplectically reduced system corresponding to initial data that lead to collisions. Lastly, we regularise the system and investigate its behaviour on near collision orbits. This involves the study of completely degenerate equilibria and the use of high-dimensional non-homogenous blow-ups. 
\medskip


\end{abstract}
 \tableofcontents 
\section{Introduction}
In the standard formulation of the $N$-body problem on any surface $M$ the phase space is taken to be $T^{\ast}(M^N\backslash \Delta)$, where $\Delta$ is the set of all $N$-tuples of points with two or more coinciding entries.  Putting it in a more straightforward way, the standard formulation of the $N$-body problem excludes any collision.

This non-compact phase space arrangement immediately begs the question of how the bodies behave when on collision orbits or in their immediate vicinity. A possible approach to these questions include finding the $\omega-$limit set of the variables of the system as it tends to a collision and the regularisation of the Hamiltonian vector field. 

Naturally, the first works of this kind discuss the Kepler problem in the plane: see \cite{frauenfelder2018restricted} for the general overview and e.g. \cite{milnor1983geometry, levi1920regularisation} for details.   

The Kepler problem on the plane is equivalent to the two-body problem. However, when the curvature of the surface on which the bodies are moving is nonzero, things are rather different. The Kepler problem becomes a separate (integrable) case, where one body, usually of mass 1 after normalisation is fixed, and the other is moving freely. In contrast, the two-body problem is not integrable and involves two freely-moving bodies of arbitrary masses.  Collision orbits, Hill's regions and near collision dynamics in the former were described in \cite{andrade2017dynamics}. Additional noteworthy result is the regularisation of the vector field near a collision point.

In \cite{borisov2005two}, Borisov and Mamaev investigate the collisions in a specific case of the two-body problem on a sphere, when the two bodies are restricted to move along a great circle. They demonstrate that the collision time is finite, and that  suitable regularisation yields elastic collisions. 

A significant percentage of the works dedicated to N-body problems on surfaces of constant curvature can be roughly divided into two lines of investigation: the more laborious approach of using the stereographic projection (e.g. 
\cite{diacu2017classical, diacu2008n}) and that of Marsden-Weinstein  reduction (\cite{arathoon2019singular, shchepetilov1998reduction, borisov2005two} and  many more). 

In this paper, we follow the latter approach: attempting to perform stereographic projection and describe collision orbits in eight dimensions proves to be a nearly insurmountable computational task. 

The most important results on which we build for our work are contained in \cite{borisov2018reduction}: namely, the method of reduction of our eight-dimensional system to a five-dimensional one through the $SO(3)$-symmetry of the setup.

The variables of the reduced system are $q$, denoting the angle between the coordinate vectors of the two bodies, $p$, the symplectic conjugate of $\dot{q}$ and $m_1,\ m_2,\  m_3$ which have the physical meaning of the angular momentum components in the body frame. 

In order to simplify our computations, we make the standard choice of an  attracting potential $V(q) = -\mu_1\mu_2\cot(q)$ describing the interaction of the particles; additionally, we demand that the two particles be identical: via scaling their masses $\mu_1$ and $\mu_2$ can both be assumed to be 1. This allows to write the system in a more palatable way, and introducing $\xi =\cot(q)$ makes it into a polynomial one, guaranteeing the smoothness of the solution curves.

\subsubsection*{Collision trajectories}
The reduced system has singularities of two types: \textit{antipodal} and \textit{collisions}, corresponding to the cases $q=\pi$ and $q=0$, respectively. We demonstrate that, for two equal masses, no trajectories leading to antipodal collisions exist (the same proof can be immediately adapted to the case of two points with arbitrary positive masses $\mu_1$ $\mu_2$, see Remark \ref{rmk3.3}). Thus, we are left with the task of investigating collisions and near collision orbits. 

 To accomplish this, we use a variety of asymptotic estimates. In a neighbourhood of a point $t^{\ast}\in\mathbb{R}\cup\{\pm\infty\}$we say that  $f \prec g$ if $f/g\to 0$ when $t\to t^{\ast}$, $f \succ g$ if $f/g\to+\infty$ and $ f\sim g $ if $f= h g$ for some bounded function $h$ that does not tend to 0.  Additional relations  $\preccurlyeq$ and $\succcurlyeq$ are the logical unions of respectively $\prec$ and $\succ$ with $\sim$. We will comment on these asymptotic estimates later on.

From the very definition of collision trajectories, we have that $q\to 0^{+}$ along a collision trajectory, thus making  $\xi = \cot(q)\to+\infty$. Henceforth, $\xi$ will act the 'measuring device' for the growth of the rest of the dynamical variables.

Leveraging on the form of the common level set of the Hamiltonian function $\mathcal{H}$ and the Casimir function $\mathcal{C}$, we  conclude that $m_3\to 0$ and the growth of $|p|$ does not exceed that of $\sqrt{\xi}$.

Introducing the quantity $I = \left<\bq_1-\bq_2, \bq_1-\bq_2\right>$, we demonstrate that the collision time is finite, as well as  that $\dot{q}$, together with   $p$, tend to $-\infty$ on collision trajectories. 

Further assessment on the growth of $p$ yields that $p\sim-\sqrt{\xi}$, which, in turn, allows us to estimate $m_3\preccurlyeq\frac{1}{\xi}$. Using this, we show that $m_2-2m_3\xi\preccurlyeq\frac{1}{\sqrt{\xi}}$, leading to a  computation proving that the projection of any collision trajectory on the sphere $m_1^2 + m_2^2 + m_3^2 = C$ has finite length (this sphere is precisely the level set of the Casimir function).

Our main result concerning the $\omega$-limit set for initial data that lead to collisions is Theorem \ref{st: theorem omega limit set}, stating that the $\omega$-limit set of collision trajectories is a singleton of the form $\{m_1=m_1^{\ast}, m_2 = m_2^{\ast}, m_3 =m_3^{\ast}, q=0, p=-\infty\}$.

Lastly, we deduce that the global behaviour of the two bodies is as follows: as they approach each other,  as a pair they move towards the plane orthogonal to the vector of angular momentum. Additionally, bodies cannot rotate  around each other infinitely many times. 

\subsubsection*{Regularisation}

The form of the reduced  equations (\ref{eq: two equal masses with xi}) suggests that the two-dimensional plane  $m_2=m_3=0$, $m_1 =\pm\sqrt{C}$ is invariant under the dynamics. The corresponding motion is the one where both bodies move  along the great circle that lies in the plane orthogonal to the vector of the angular momentum. We refer to this setup as the \textit{motion in the invariant plane}. Unless the bodies are in a state of a relative equilibrium on the opposite sides of the sphere (an antipodal singularity that we don't consider), a collision is inevitable. Additionally, this case illustrates that, in contrast with the Kepler problem, collisions can happen at any values of the momentum. 

This system reduces to a two-dimensional one, and through a change of variables and introduction of a fictitious time we regularise the equations. The only equilibrium of the new system is the origin, and we use non-homogeneous blow-up to investigate the trajectories nearby. 

For the general case, we observe that after an appropriate change of variables, the  equilibria of the system form a two-dimensional plane; this entails that we only need to perform a three-dimensional nonhomogeneous blow-up, as opposed to a five-dimensional one, which simplifies the problem greatly. 

Thus two of the eigenvalues of the linearised matrix at newly obtained equilibria will be identically zero, but this is inconsequential, since the dynamics in the corresponding directions is trivial (a plane of equilibria). However,  since in the blow-up we perform the exceptional divisor is a two dimensional sphere, we need to utilise two charts to cover it completely. This leads us to use two different non-homogenous blow-ups in order to provide two charts. 

We determine the number and types of equilibria on both charts and demonstrate that equilibria of the same type, when twice covered, correspond to the same directions in the original variables.  Lastly, the dynamics of all variables near the exceptional divisor are described. 

{\bf Acknowledgments} The authors would like to thank the anonymous referee for fruitful suggestions and having helped to improve the readability of the paper.

\section{Setup}
The equations of motion of two free bodies on the sphere with masses $\mu_1$ and $\mu_2$ that act on one another with an attracting force have the form (\cite{diacu2012n}, \cite{diacu2012n2})
\begin{equation}
    \label{eq:general one}
    \begin{cases}
    \dot{ \mathbf{q}}_1=\frac{\mathbf{p}_1}{\mu_1};\\
    \dot{ \mathbf{q}}_2=\frac{\mathbf{p}_2}{\mu_2};\\
    \dot{ \mathbf{p}}_1=V\left(||\mathbf{q}_1-\mathbf{q}_2||\right)_{\mathbf{q}_1} - \frac{\left<\bp_1,\bp_1\right>}{\mu_1}\bq_1;\\
    \dot{ \mathbf{p}}_2=V\left(||\mathbf{q}_1-\mathbf{q}_2||\right)_{\mathbf{q}_2} - \frac{\left<\bp_2,\bp_2\right>}{\mu_2}\bq_2;\\
    ||\mathbf{q}_1|| = ||\mathbf{q}_2|| = 1;\\
    \mathbf{p}_1\cdot\mathbf{q}_1 =\mathbf{p}_2\cdot\mathbf{q}_2 = 0. 
    \end{cases}
\end{equation}
with the energy function (Hamiltonian) $\mathcal{H}$ given by
\begin{equation}
    \label{eq:Hamiltonian gen form}
    \mathcal{H} = \frac{||\bp_1||^2}{2\mu_1} + \frac{||\bp_2||^2}{2\mu_2} + V
    (||\bq_1-\bq_2||). 
\end{equation}
The problem has obvious $SO(3)$-symmetry, and, consequently, at least three conserved quantities: the components of the total angular momentum $\mathbf{L}$, of explicit form
\[
\mathbf{L} = \bq_1\times\bp_1 =\bq_2\times\bp_2.
\]
These quantities are the components of the equivariant momentum map. 

We assume that the potential $V\left(||\bq_1-\bq_2||\right)$ is an attracting one, meaning that $V\left(||\bq_1-\bq_2||\right)\to -\infty$ when $\bq_1\to \bq_2$ and  $V\left(||\bq_1-\bq_2||\right)\to +\infty$ when $\bq_1\to- \bq_2$. 

Such a choice of the potential immediately entails that the phase space of this problem is $\mathcal{Q}:=T^{\ast}(S^2\times S^2\backslash\Delta)$ (taken with the standard cotangent bundle symplectic structure), where $\Delta$, referred to in the literature as 'the big diagonal' is the set consisting of pairs of the same or antipodal points on the sphere.

\begin{figure}
\centering

\tikzset{every picture/.style={line width=0.75pt}} 

\begin{tikzpicture}[x=0.75pt,y=0.75pt,yscale=-1,xscale=1]

\draw    (282.11,159.87) -- (281.12,10.6) ;
\draw [shift={(281.11,8.6)}, rotate = 89.62] [color={rgb, 255:red, 0; green, 0; blue, 0 }  ][line width=0.75]    (10.93,-3.29) .. controls (6.95,-1.4) and (3.31,-0.3) .. (0,0) .. controls (3.31,0.3) and (6.95,1.4) .. (10.93,3.29)   ;
\draw    (282.11,159.87) -- (197.99,210.76) ;
\draw [shift={(196.28,211.79)}, rotate = 328.83] [color={rgb, 255:red, 0; green, 0; blue, 0 }  ][line width=0.75]    (10.93,-3.29) .. controls (6.95,-1.4) and (3.31,-0.3) .. (0,0) .. controls (3.31,0.3) and (6.95,1.4) .. (10.93,3.29)   ;
\draw    (282.11,159.87) -- (393.11,160.73) ;
\draw [shift={(395.11,160.75)}, rotate = 180.44] [color={rgb, 255:red, 0; green, 0; blue, 0 }  ][line width=0.75]    (10.93,-3.29) .. controls (6.95,-1.4) and (3.31,-0.3) .. (0,0) .. controls (3.31,0.3) and (6.95,1.4) .. (10.93,3.29)   ;
\draw    (282.11,159.87) -- (202.01,133.23) ;
\draw [shift={(200.11,132.6)}, rotate = 18.4] [color={rgb, 255:red, 0; green, 0; blue, 0 }  ][line width=0.75]    (10.93,-3.29) .. controls (6.95,-1.4) and (3.31,-0.3) .. (0,0) .. controls (3.31,0.3) and (6.95,1.4) .. (10.93,3.29)   ;
\draw    (282.11,159.87) -- (342.92,227.27) ;
\draw [shift={(344.26,228.75)}, rotate = 227.94] [color={rgb, 255:red, 0; green, 0; blue, 0 }  ][line width=0.75]    (10.93,-3.29) .. controls (6.95,-1.4) and (3.31,-0.3) .. (0,0) .. controls (3.31,0.3) and (6.95,1.4) .. (10.93,3.29)   ;
\draw    (282.11,159.87) -- (316.81,93.79) ;
\draw [shift={(317.74,92.02)}, rotate = 117.7] [color={rgb, 255:red, 0; green, 0; blue, 0 }  ][line width=0.75]    (10.93,-3.29) .. controls (6.95,-1.4) and (3.31,-0.3) .. (0,0) .. controls (3.31,0.3) and (6.95,1.4) .. (10.93,3.29)   ;
\draw    (282.11,159.87) -- (293.92,36.59) ;
\draw [shift={(294.11,34.6)}, rotate = 95.47] [color={rgb, 255:red, 0; green, 0; blue, 0 }  ][line width=0.75]    (10.93,-3.29) .. controls (6.95,-1.4) and (3.31,-0.3) .. (0,0) .. controls (3.31,0.3) and (6.95,1.4) .. (10.93,3.29)   ;
\draw  [draw opacity=0] (208.1,217.42) .. controls (193.81,189.23) and (193.17,141.62) .. (208.52,97.77) .. controls (227.98,42.18) and (265.78,13.7) .. (292.95,34.15) .. controls (314.84,50.63) and (323.15,93.86) .. (315.37,139.12) -- (257.71,134.8) -- cycle ; \draw   (208.1,217.42) .. controls (193.81,189.23) and (193.17,141.62) .. (208.52,97.77) .. controls (227.98,42.18) and (265.78,13.7) .. (292.95,34.15) .. controls (314.84,50.63) and (323.15,93.86) .. (315.37,139.12) ;  
\draw  [draw opacity=0] (373.28,150.54) .. controls (388.77,155.86) and (401.05,162.78) .. (408.23,170.94) .. controls (429.93,195.59) and (396.84,222.16) .. (334.32,230.28) .. controls (271.79,238.4) and (203.51,225) .. (181.81,200.35) .. controls (160.11,175.7) and (193.2,149.13) .. (255.72,141.01) .. controls (297.46,135.59) and (341.77,139.76) .. (373.07,150.46) -- (295.02,185.64) -- cycle ; \draw   (373.28,150.54) .. controls (388.77,155.86) and (401.05,162.78) .. (408.23,170.94) .. controls (429.93,195.59) and (396.84,222.16) .. (334.32,230.28) .. controls (271.79,238.4) and (203.51,225) .. (181.81,200.35) .. controls (160.11,175.7) and (193.2,149.13) .. (255.72,141.01) .. controls (297.46,135.59) and (341.77,139.76) .. (373.07,150.46) ;  
\draw  [dash pattern={on 0.84pt off 2.51pt}]  (282.11,159.87) -- (208.1,217.42) ;
\draw  [dash pattern={on 0.84pt off 2.51pt}]  (282.11,159.87) -- (315.11,140.6) ;
\draw  [dash pattern={on 0.84pt off 2.51pt}]  (282.11,159.87) -- (265.11,140.6) ;
\draw  [draw opacity=0] (265.19,139.7) .. controls (264.57,127.92) and (265.87,116.78) .. (269.35,107.42) .. controls (280.29,78.01) and (308.63,77.44) .. (332.65,106.14) .. controls (356.66,134.85) and (367.27,181.96) .. (356.33,211.37) .. controls (353.41,219.22) and (349.25,225.02) .. (344.26,228.75) -- (312.84,159.4) -- cycle ; \draw   (265.19,139.7) .. controls (264.57,127.92) and (265.87,116.78) .. (269.35,107.42) .. controls (280.29,78.01) and (308.63,77.44) .. (332.65,106.14) .. controls (356.66,134.85) and (367.27,181.96) .. (356.33,211.37) .. controls (353.41,219.22) and (349.25,225.02) .. (344.26,228.75) ;  
\draw    (282.11,159.87) -- (358.94,180.35) ;
\draw [shift={(360.87,180.86)}, rotate = 194.92] [color={rgb, 255:red, 0; green, 0; blue, 0 }  ][line width=0.75]    (10.93,-3.29) .. controls (6.95,-1.4) and (3.31,-0.3) .. (0,0) .. controls (3.31,0.3) and (6.95,1.4) .. (10.93,3.29)   ;
\draw  [draw opacity=0] (314.06,195.56) .. controls (303.21,198.78) and (290.24,200.48) .. (276.35,200.1) .. controls (259.02,199.62) and (243.3,195.99) .. (231.46,190.41) -- (277.27,166.84) -- cycle ; \draw   (314.06,195.56) .. controls (303.21,198.78) and (290.24,200.48) .. (276.35,200.1) .. controls (259.02,199.62) and (243.3,195.99) .. (231.46,190.41) ;  
\draw   (291.82,196.21) .. controls (295.85,197.67) and (299.78,198.35) .. (303.6,198.23) .. controls (299.89,199.14) and (296.28,200.84) .. (292.79,203.34) ;
\draw  [draw opacity=0] (323.25,171.08) .. controls (322.23,174.33) and (320.64,177.46) .. (318.45,180.32) .. controls (315.52,184.15) and (311.85,187.08) .. (307.8,189.06) -- (294.61,162.1) -- cycle ; \draw   (323.25,171.08) .. controls (322.23,174.33) and (320.64,177.46) .. (318.45,180.32) .. controls (315.52,184.15) and (311.85,187.08) .. (307.8,189.06) ;  
\draw   (309.53,185.05) .. controls (313.5,183.43) and (316.87,181.29) .. (319.62,178.63) .. controls (317.48,181.8) and (315.95,185.49) .. (315.04,189.68) ;
\draw  [draw opacity=0] (223.09,139.77) .. controls (223.67,127.79) and (226.88,114.57) .. (232.9,101.55) .. controls (244.56,76.32) and (263.56,58.65) .. (281.28,54.3) -- (271.72,119.49) -- cycle ; \draw   (223.09,139.77) .. controls (223.67,127.79) and (226.88,114.57) .. (232.9,101.55) .. controls (244.56,76.32) and (263.56,58.65) .. (281.28,54.3) ;  
\draw   (247.39,72.62) .. controls (251.48,71.3) and (255,69.44) .. (257.95,67.01) .. controls (255.57,70) and (253.76,73.55) .. (252.52,77.66) ;
\draw  [fill={rgb, 255:red, 0; green, 0; blue, 0 }  ,fill opacity=1 ] (292.24,34.6) .. controls (292.24,33.56) and (293.07,32.73) .. (294.11,32.73) .. controls (295.14,32.73) and (295.98,33.56) .. (295.98,34.6) .. controls (295.98,35.63) and (295.14,36.47) .. (294.11,36.47) .. controls (293.07,36.47) and (292.24,35.63) .. (292.24,34.6) -- cycle ;
\draw  [fill={rgb, 255:red, 0; green, 0; blue, 0 }  ,fill opacity=1 ] (198.24,132.6) .. controls (198.24,131.56) and (199.07,130.73) .. (200.11,130.73) .. controls (201.14,130.73) and (201.98,131.56) .. (201.98,132.6) .. controls (201.98,133.63) and (201.14,134.47) .. (200.11,134.47) .. controls (199.07,134.47) and (198.24,133.63) .. (198.24,132.6) -- cycle ;

\draw (260,7.95) node [anchor=north west][inner sep=0.75pt]   [align=left] {$z$};
\draw (175,210.72) node [anchor=north west][inner sep=0.75pt]   [align=left] {$x$};
\draw (395,145.44) node [anchor=north west][inner sep=0.75pt]   [align=left] {$y$};
\draw (345,230) node [anchor=north west][inner sep=0.75pt]   [align=left] {$x''$};
\draw (320,75.88) node [anchor=north west][inner sep=0.75pt]   [align=left] {$y'$};
\draw (180,102.86) node [anchor=north west][inner sep=0.75pt]   [align=left] {$-z'$};
\draw (365,180) node [anchor=north west][inner sep=0.75pt]   [align=left] {$x'$};
\draw (297,15) node [anchor=north west][inner sep=0.75pt]   [align=left] {$\displaystyle \mu _{2}$};
\draw (259.36,205.65) node [anchor=north west][inner sep=0.75pt]   [align=left] {$\displaystyle \psi $};
\draw (324.86,179.72) node [anchor=north west][inner sep=0.75pt]   [align=left] {$\displaystyle \phi $};
\draw (235.84,45.37) node [anchor=north west][inner sep=0.75pt]   [align=left] {$\displaystyle \pi- \theta $};
\draw (176.27,124.37) node [anchor=north west][inner sep=0.75pt]   [align=left] {$\displaystyle \mu _{1}$};

\end{tikzpicture}

\caption{Euler angles and the body frame. Notation $-z'$ signifies that the $z'$-axis of the moving system points in the opposite direction.}
\label{fig: two coord systems}
\end{figure}
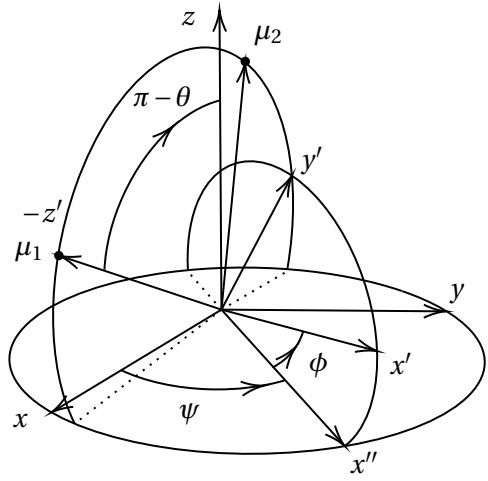
However, $\mathcal{Q}$ is 8-dimensional and equations (\ref{eq:general one}) are a rather  complicated system. In order to simplify the problem, one might use symplectic reduction with respect to the aforementioned $SO(3)$-symmetry of the problem. This reduction was accomplished by Borisov, Mamaev, Montaldi and Garc\'ia-Naranjo in \cite{borisov2018reduction}.  We give a brief description of if below; for more details and  uses of this method, see \cite{borisov2018reduction}, \cite{garcia2021attracting} and \cite{balabanova2021two}.

To 'separate' the symmetric action of the group, we re-parameterise the phase space. We introduce the new 'body frame' with axes $x',\ y', \ z'$ ( $x,y,z$ denoting the old coordinate axes ) that moves with the two bodies in the following way: the coordinate vector of the first body always coincides with the $z'$-axis of the new system, and the basis changes in such a way that the second body is always in $z'-y'$ plane. We denote by $q$ the angle between the two coordinate vectors and without loss of generality we may assume that the initial positions of the two masses are, respectively, $\mathbf{x}_1=(0,0,-1)$ and $\mathbf{x}_2 = (0,\sin(q), -\cos(q)) $. 

As we have stated, the form of our potential prohibits collisions as well as the two bodies from occupying antipodal points on the sphere. Hence, $q\in I:=(0,\pi)$. We suppose that  $g\in SO(3)$ and $\zeta\in TSO(3)$.

Since  $TSO(3)$ can be trivialised as $ SO(3)\times\mathfrak{so}(3)$, we can write
\begin{equation}
\label{eq: spherical coords two particleS}
    \begin{split}
       & TI\times TSO(3)\to T\mathcal{Q}\\
        &(q,\dot{q}, \theta,\phi, \psi,\boldsymbol\omega_1,\boldsymbol\omega_2,\boldsymbol\omega_3)\mapsto(g\cdot \mathbf{x}_1, g\cdot \mathbf{x}_2, g\zeta\cdot\mathbf{x}_1, g\zeta\cdot\mathbf{x}_2 + g\mathbf{x}_2'\dot{q}),
    \end{split}
\end{equation}
with $g$ expressed through the three Euler angles as 
\begin{footnotesize}
\begin{equation} \label{eq:Euler angles}
g = \left( \begin{array}{c c c}
    \cos(\phi)\cos(\psi) - \cos(\theta)\sin(\psi)\sin(\phi)& -\sin(\phi)\cos(\psi)-\cos(\theta)\sin(\psi)\cos(\phi)& \sin(\theta)\sin(\psi)\\
    \cos(\phi)\sin(\psi) + \cos(\theta)\cos(\psi)\sin(\phi)& -\sin(\phi)\sin(\psi)+\cos(\theta)\cos(\psi)\cos(\phi)& -\sin(\theta)\cos(\psi)\\
   \sin(\theta)\sin(\phi)& \sin(\theta)\cos(\phi)& \cos(\theta)
    \end{array}\right),
\end{equation}
    \end{footnotesize}
see Figure \ref{fig: two coord systems} for illustration. 
The variable  $\zeta$ has  the physical meaning of  angular velocity in the body frame (\cite{marsden2013introduction}).
After this change the Hamiltonian turns into 
\begin{small}
\begin{equation}
\label{eq:hamilt red}
\begin{split}
\mathcal{H} \ = \ & \frac{1}{2\mu_1\mu_2 }\biggl(
\mu_2  \left((m_1-p)^2+m_2^2\right)+m_3 \left(-2 \mu_2  m_2 \cot (q)+\mu_1  m_3 \csc^2q+\mu_2  m_3 \cot^2q\right)\biggr)+{}\\&\qquad{}+\frac{ p^2}{2 \mu_2 } +  V(q).
\end{split}
\end{equation}\end{small}
The  non-zero Poisson brackets in the reduced variables are given by 
    \begin{align}
    \label{Poiss}
        &\{m_1,m_2\} = -m_3,\ &  &\{m_2,m_3\} = -m_1 ,\nonumber\\ &\{m_1,m_3\} = m_2, \ &   \   &\{q,p\} = 1.
    \end{align}
    The three invariant values 'collapse' into one Casimir function
    \begin{equation}
        \label{eq: Casimir}
        \mathcal{C} = m_1^2  + m_2^2 + m_3^2.
    \end{equation}

    For the purposes of this work, we choose $V(q)= -\mu_1\mu_2\cot(q)$, where $q$ is the angle between the two coordinate vectors of the particles. This is a standard choice of potential for the $N$-body problem on a sphere (\cite{borisov2005superintegrable}) (and its hyperbolic analogue, on a Lobachevsky plane), arising as well as one of the  solutions of the generalized Bertrand problem of finding potentials depending only on geodesic distance and with closed orbits  \cite{kozlov1992kepler}. 

    With these assumptions made, the five-dimensional equations of motion have the form
    \begin{equation}
        \label{eq: reduced two bodies}
    \begin{cases}
    \dot{m}_1 = \frac{1}{\mu_1 \mu_2}(-m_2 m_3 \mu_2+\mu_2 \cot(q) (-m_2^2+m_3^2+m_2 m_3 \cot(q))+m_2 m_3 \mu_1 \csc(q)^2)\\\dot{m}_2 = \frac{1}{(\mu_1 \mu_2)}(m_3 (2 m_1-p) \mu_2+m_1 m_2 \mu_2 \cot(q)-m_1 m_3 (\mu_1+\mu_2) \csc(q)^2)\\\dot{m}_3 = \frac{1}{\mu_1}(m_2 p-m_1 m_3 \cot(q))\\\dot{q}  = \frac{1}{\mu_1 \mu_2}(-m_1 \mu_2+p (\mu_1+\mu_2))\\\dot{p} = \frac{1}{\mu_1 \mu_2}((-\mu_2 (m_2 m_3+\mu_1^2 \mu_2)+m_3^2 (\mu_1+\mu_2) \cot(q)) \csc(q)^2)
    
    \end{cases}
    \end{equation}
    We consider the specific case with two identical bodies for simplicity, but our analysis can be carried out, `mutatis mutandis', for general values of the masses; without loss of generality we may assume that $\mu_1 = \mu_2 = 1$, which turns (\ref{eq: reduced two bodies}) into
    
\begin{equation}
\label{eq: two equal masses} 
\begin{cases}
\dot{m}_1 = \cot(q)\left(-m_2^2 + m_3^2 + 2m_2m_3\cot(q)\right),\\
\dot{m}_2 = m_1m_2\cot(q)-m_3(-2m_1 + p + 2m_1\csc^2(q)),\\
\dot{m_3} = m_2p  - m_1m_3\cot(q),\\
\dot{q}  = 2p-m_1,\\
\dot{p} = -\csc^2(q)(1  +m_2m_3 - 2m_3^2\cot(q)).
\end{cases}
\end{equation}
We introduce the two classes of orbits we are focused on in this work:
\begin{definition}
If $\bq_1(t)-\bq_2(t)\to 0$, we call such a trajectory a \textbf{collision trajectory} or a \textbf{collision singularity}. If $\bq_1(t)+\bq_2(t)\to0$, we refer to such a trajectory as an \textbf{antipodal singularity}. The encompassing term for the trajectories of both kinds is \textbf{singular trajectories} 
\end{definition}

\section{Singular trajectories}
\subsection{Omega limit sets of collision singularities}
In this section we study singular trajectories (in particular collision trajectories) using reduced variables. The goal of this section is to show that for an initial datum $D$ that corresponds to collision trajectories, the omega-limit set $\omega(D)$ is a singleton, for which the reduced coordinates are $\{m_1=m_1^*, m_2=m_2^*, m_3=m_3^*=0, q=0, p=-\infty\}.$ Recall that the $\omega$-limit set of a solution $x(t,D)$ of a vector field $X$ in $\mathbb{R}^n$ with initial data $x(0,D)=D$ is given by 
$$\omega(D):=\{p\in \mathbb{R}^n \;  \text{ such that } \; \exists \{t_k\}_{k\in \mathbb{N}},\; t_k\nearrow t^* \text{ with } \; \lim_{k\rightarrow \infty}x(t_k, D)=p\},$$
where $t^*$ is in our case the collision time. 

While proving this, we will show that the time to collision is finite and that antipodal singularities can never occur.  

In general it is very difficult to determine the omega-limit set of a solution of a differential equation (see for instance \cite{arsie2010locating} for a general method to determine it).

To recast  system (\ref{eq: two equal masses}) in a form more amenable to treatment we start with the following observation the proof of which is immediate: 

\begin{lemma}\label{st:lemma1}
Consider an ODE in an open subset $U$ of \  $\mathbb{R}^k$ with coordinates $x^1, \dots, x^k$, $\dot x^i=f^i(x^1, , \dots, x^k),$ $i=1,\dots, k$. Suppose this vector field admits $l$ conserved quantities given by smooth functions $C_1, \dots C_l$. Let $\lambda: x\mapsto y(x)$ be a diffeomorphism on $U$. Then the system of ODE written with respect to the $y$ coordinates admits $l$ conserved quantities described by the function $C_1, \dots, C_l$ written in the $y$ coordinates.
\end{lemma}
In particular, the coordinates transformation used in Lemma \ref{st:lemma1} does not need to be canonical.

Lemma \ref{st:lemma1} allows us to introduce a new variable and make the system (\ref{eq: two equal masses}) a polynomial one, while keeping both $\mathcal{H}$ and $\mathcal{C}$ as invariant functions, albeit with a coordinate transformation. Let $\xi = \cot(q)$. Then $\dot{\xi} = -\csc^2(q), \ \dot{q} = -(\xi^2 + 1)\dot{q}$, and the system (\ref{eq: two equal masses}) turns into 
\begin{equation}
    \label{eq: two equal masses with xi}
    \begin{cases}
    \dot{m}_1 = \xi\left(-m_2^2 + m_3^2 + 2m_2m_3\xi\right),\\
\dot{m}_2 = m_1m_2\xi-m_3p - 2m_1m_3\xi^2,\\
\dot{m_3} = m_2p  - m_1m_3\xi,\\
\dot{\xi}  = - (\xi^2 + 1)( 2p-m_1),\\
\dot{p} = -(\xi^2 + 1)(1  +m_2m_3 - 2m_3^2\xi),
    \end{cases}
\end{equation}
with the energy function
\begin{equation}
    \label{eq: the Ham}
    \mathcal{H} = \frac12\left(m_1^2 + m_2^2  - 2m_1p + 2p^2 + \xi(-2-2m_2m_3 + m_3^2\xi) + m_3^2(1 + \xi^2)\right).
\end{equation}
The Casimir $\mathcal{C}$, clearly, does not change. 

By definition, collision orbits  have $q(t)\to0$, and, consequently, $\xi(t)\to+\infty$. However, the behaviours of the rest of the variables on collision trajectories are not immediately clear. 

On the common level set of the Hamiltonian $\mathcal{H} = h$ and the Casimir $\mathcal{C} = C$, we have 
\begin{equation}
\label{eq:Ham and Cas common level set}
2p^2 - 2m_1p + 2m_3^2\xi^2 - 2\xi(1 + m_2m_3) = 2h-C
\end{equation}
or
\begin{equation}
    \label{eq: estimate p}
2\left(p - \frac{m_1}{2}\right)^2  + \left(\sqrt{2}m_3\xi-\frac{m_2}{\sqrt{2}}\right)^2  - 2\xi= 2h-C + \frac{m_1^2 + m_2^2}{2},
\end{equation}
\begin{lemma}
There are no trajectories leading to antipodal singularities.
\end{lemma}
\begin{proof}
On a trajectory leading to an antipodal singularity, $q\to\pi^-$ and, consequently, $\xi\to-\infty$. The left hand side of (\ref{eq: estimate p}) must, therefore, be unbounded, in contrast with the right hand side.
\end{proof}
\begin{rmk}\label{rmk3.3}
    One can easily check through completing the squares in (\ref{eq:hamilt red})   in the spirit of (\ref{eq: estimate p}) that this lemma holds for any two arbitrary positive values of $\mu_1$ and $\mu_2$, provided the potential is cotangent; the expression in question will be 
    \begin{equation*}
\begin{split}
&\left(\sqrt{\mu_1 + \mu_2}p-\frac{\mu_2}{\sqrt{\mu_1 + \mu_2}}m_1\right)^2 + \left(\sqrt{\mu_1 + \mu_2}\xi m_3 - \frac{\mu_2}{\sqrt{\mu_1 + \mu_2}}m_2\right)^2 - 2\mu_1^2\mu_2^2\xi=\\=& \  2\mu_1\mu_2h - \mu_2C + \frac{\mu_2^2}{\mu_1 + \mu_2}(m_1^2+ m_2^2) + (\mu_2-\mu_1)m_3^2
\end{split}
    \end{equation*}
\end{rmk}

\begin{lemma}
On the trajectories tending to a collision, $\xi\to +\infty$ and $m_3\to0$.
\end{lemma}
\begin{proof}
The first statement has been mentioned before and is clear.  The second one immediately follows from (\ref{eq: estimate p}): if $m_3\not\to 0$, then there exists a sequence of times $\{t_k\}$ converging to collision time along which the left hand side of \eqref{eq: estimate p} grows as $\xi^2$ and is therefore unbounded, while the right hand side remains bounded. \end{proof}

The main tool used in this section is a sort of growth analysis of the variables involved as the particles approach collision. Let us remark that since the vector field \eqref{eq: two equal masses with xi} is polynomial, the solutions curves are smooth. In order to proceed, we need to introduce some notation. 

\begin{definition}\label{def:asymptotic}
Consider two smooth functions $f(t)$ and $g(t)$ of time. Let $t^*$ be the collision time. Then we write
\begin{enumerate}
    \item $f\succ g$ with $t\to t^{\ast}\in\bar{\mathbb{R}}$ if $\lim\limits_{t\to t^{\ast}}\frac{f}{g} = \infty$. In this case, in a suitable left neighbourhood of $t^*$ we can represent $f(t)=h(t)g(t)$, where $h(t)$ is a smooth function such that $\lim\limits_{t\to t^{\ast}}h(t)=\infty.$
     \item $f\prec g$ with $t\to t^{\ast}\in\bar{\mathbb{R}}$ if $\lim\limits_{t\to t^{\ast}}\frac{f}{g} = 0$.  In this case, in a suitable left neighbourhood of $t^*$ we can represent $f(t)=h(t)g(t)$, where $h(t)$ is a smooth function such that $\lim\limits_{t\to t^{\ast}}h(t)=0.$
     \item $f\sim g$ if there exists a left neighbourhood of $t^*$, where  $f = h g$, with $h$ a bounded function that \textbf{does not} tend to 0. (Of course, there can be a sequence of times $\{t_k\}$ converging to $t^*$ from below such that $\lim\limits_{k\to +\infty}h(t_k)=0$.) 
     \item if $f\sim g$ or $f\succ g$, we write $f\succcurlyeq g$; in the similar case for $f\prec g$, the notation is $f\preccurlyeq g$;
\end{enumerate}

When there is no need to specify an bounded function, we denote it by $\overline{O}$. Similarly, $\overline{O}(x^{\alpha})$ represents a bounded function $\overline{O}$ multiplied by $x^{\alpha}$.
\end{definition}
\begin{rmk}
The relations introduced in Definition \ref{def:asymptotic} are not a total order in general, since not all functions are comparable: given a function, we can create one incomparable to it by multiplying it by unbounded oscillating function. For instance, $f(t)=t$ and $g(t)=t^2(\cos(t)+1)+1$ or $g(t)=t^2(\cos(t)+1)$ are incomparable to each other. 
\end{rmk}
The $\sim$ relation in Definition \ref{def:asymptotic} is not an equivalence: since we are allowing the function $h$ to assume zero values, $\sim$ is not  symmetric nor transitive in general. Nonetheless, in certain cases it has some nice  properties:
\begin{lemma}
\label{st: lemma about the relations}
Consider relations $\sim$ and $\preccurlyeq$ from Definition \ref{def:asymptotic}. The following hold:
\begin{enumerate}
\item the relation $\preccurlyeq$ is transitive;
\item if $f=hg$ where $h$ is bounded and bounded away from zero in a left-neighbourhood of $t^*$, then the relation $f\sim g$ is equivalent to $g\sim f$;
    \item suppose $f\sim g$ with $f = f_1g$ and $g\sim h$ with $g = g_1h$. Then if $\lim\limits_{t\to t^{\ast}}g_1 = C_1\ne 0 $ or $\lim\limits_{t\to t^{\ast}}h_1 = C_2\ne 0 $, where $C_1,\ C_2$ are constants, $f\sim h$.  
\end{enumerate}
\begin{proof}
The  second entry is straightforward; we need to discuss the first and the third.  In order to prove  the third, we observe that from the equalities in point (3), we have $f = f_1g_1h$, with $f_1$ or $g_1$ bounded away from zero. If one of them tends to a nonzero constant, the product can't tend to zero either. 

To demonstrate the first, we stress that our definition of the relation $f\preccurlyeq g$ means only that $g$ is $f$ multiplied by a bounded function: we don't really know anything else. 

So if $f\preccurlyeq g$ and $g\preccurlyeq f$, then there exists a left neighbourhood of $t^{\ast}$ and two bounded functions $f_1$ and $g_1$, such that $f = f_1g$ and $g = g_1h$, meaning that $f = f_1g_1h$. Now, $f_1g_1$ is a bounded function, and that is all we need to know in order to assert that $f\preccurlyeq h$. 
\end{proof}
\end{lemma}
A preliminary estimate on the velocity of decline of $m_3$ is given by the following:

\begin{lemma}
On the trajectories tending to collision, $m_3\preccurlyeq\xi^{-\frac12}$. 
\end{lemma}

\begin{proof}
Assume for the sake of contradiction that this is not true. This entails that $m_3$ can be represented as $ m_3 = g(t)\xi^{-\frac12}$, where $g(t)$ is a smooth function which is unbounded in some suitable left neighbourhood of $t^*$. In particular $g(t)$ might converge to $\infty$ for $t\to t^*$ from below (in which case $m_3\succ \xi^{-\frac12}$), or be simply unbounded.  

Under this assumption, (\ref{eq: estimate p}) turns into
\begin{equation}
\label{eq: est m3}
2\left(p -\frac{m_1}{2}\right)^2 
+ 2\xi\left(g^2(t)-1\right) - 2m_2g(t)\sqrt{\xi} = 2h-C  +\frac{m_1^2}{2}.
\end{equation}
Assume that $|g(t)|\to +\infty$ on a sequence $\{t_n\},$ as $n\to+\infty$. Then 
\[
2\left(p(t_n) -\frac{m_1(t_n)}{2}\right)^2  + 2\left(\sqrt{\xi(t_n)}\sqrt{g^2(t_n)-1} - \frac{m_2(t_n)}{2}\frac{g(t_n)}{\sqrt{g^2(t_n)-1}}\right)^2 = 2h-C  +\frac{m_1^2}{2} + \frac{m_2^2}{2}\frac{g^2(t_n)}{g^2(t_n)-1}
\]
The left hand side is unbounded and assumes large positive values, whereas the right hand side is bounded (since it can be easily seen that, for example, $\frac{g^2(t_n)}{g^2(t_n)-1}<2$), which leads to a contradiction. 
Therefore, $m_3\preccurlyeq\xi^{-\frac12}$.
\end{proof}

\begin{rmk}
\label{st:remark about p}
From the statement above we can demonstrate that  $|p|\preccurlyeq\sqrt{\xi}$. Indeed, looking at \eqref{eq: estimate p}, when $m_3\prec \xi^{-\frac12}$, the expression $\left(\sqrt{2}m_3\xi - \frac{m_2}{\sqrt{2}}\right)^2$ is 'overpowered' by $-2\xi$, and $|p|\sim \sqrt{\xi}$. 

When $m_3\sim\xi^{-\frac12}$, $m_3$ can be rewritten as $m_3 = \frac{\tilde{g}(t)}{\sqrt{\xi}}$, where $\tilde{g}(t)$ is some bounded function that doesn't tend to 0, giving us an expression of exactly the same form (barring $\tilde{g}$ instead of  $g$) as (\ref{eq: est m3}). 
If  $\tilde{g}(t)\not\to\pm 1$, $|p|\sim\sqrt{\xi}$. If $\tilde{g}(t)\to\pm 1$, the coefficient at $\xi$ slows down the growth, and $p^2$ has to increase slower than $\xi$ as well.
\end{rmk}

Knowing this, we can demonstrate the following:

\begin{lemma}\label{lemma:finite time}
Collisions happen in finite time, i.e. if $q\to 0$ ($\mathbf{q}_1 - \mathbf{q}_2\to0$), \ $t\to t^{\ast}<+\infty$.
\end{lemma}
\begin{proof}

We introduce the quantity $I:=\left<\bq_1-\bq_2,\bq_1-\bq_2\right>$.  Evidently, $I\to 0 $ on collision trajectories. 

From applying the theorem of sines to the isosceles triangle with the base $\sqrt{I}$ and two sides of length 1 with separating angle $q$,  $I = 4\sin^2\left(\frac{q}{2}\right)$. Therefore, $\dot{I} = 2\sin(q)\dot{q}$ and $\ddot{I} = 2\cos(q)\left(\dot{q}\right)^2 + 2\sin(q)\ddot{q}$. Substituting the expressions for $\ddot{q}$ and $\dot{q}$ from (\ref{eq: two equal masses}), we get
\begin{equation}
    \label{eq: I ddot}
    \begin{split}
        \ddot{I} &= 2\cos(q)\left(\dot{q}\right)^2 + 2\sin(q)(2\dot{p}-\dot{m}_1) = 2\cos(q)(2p-m_1)^2 -\frac{4}{\sin(q)}\left(1 +
         m_2m_3-2m_3^2\cot(q)\right) -\\&- 2\cos(q)\left(-m_2^2 + m_3^2 + 2m_2m_3\cot(q)\right)
    \end{split}
\end{equation}
In the expression above, we can replace $\cot(q)$ with $\xi$. Representing $q$ as $\mathrm{arccot}(\xi)$ and taking the Taylor series at $\xi=+\infty$ give  $\cos(q) = 1 - \overline{O}\left(\frac{1}{\xi^2}\right)$ and $\frac{1}{\sin(q)} = \xi + \overline{O}\left(\frac{1}{\xi}\right)$. Recalling Remark \ref{st: behaviour of p} and boundedness of $m_i$, we may write
\begin{equation}
    \label{eq: more on I ddot}
    \begin{split}
    (\ref{eq: I ddot})&= 2(2p-m_1)^2 - 4\xi\left(1 + m_2m_3 - 2m_3^2\xi\right) - 4m_2m_3\xi + \overline{O} = \\&= 8\left(p-\frac{m_1}{2}\right)^2   + 8m_3^2\xi^2 - 4\xi(1 + m_2m_3) - 4m_2m_3\xi + \overline{O} =\\&= 8h-4C + 4\xi(1 + m_2m_3) - 4m_2m_3\xi + \overline{O} = 4\xi + \overline{O}\to+\infty
    \end{split}
\end{equation}

Therefore, $\ddot{I} = 4\xi + \overline{O}\to+\infty$. For a contradiction, we assume that the collision time is infinite; on the other hand, we can find some $t_0$ such that $\ddot{I}(t)>K>>0$ for $t>t_0$. Thus, $I(t)>\frac{K}{2}t^2 
+ Mt + L$ for some constants $M,L$, when $t$ tends to collision time, contradicting $I\to 0$ if $t^*=+\infty.$

\end{proof}

We explore further the behavior of the dynamics variables along collision solutions:
\begin{proposition}
\label{st: behaviour of p}
As $t\to t^{\ast}$, $\dot{q}\to-\infty$ and $p\to-\infty$.
\end{proposition}
\begin{proof}
We first focus on the behavior of $\dot q$. As a start, we have that $\ddot{I}$ is a positive function when $t$ is greater than some $t_0$ sufficiently close to the collision time, as it is clear from the proof of Lemma \ref{lemma:finite time}. Therefore, $\dot{I}$ is a monotonous increasing function for corresponding values of $t$. In the light of this, since  by definition $I\geq 0$ and along collisions $I\to 0$, $\dot{I}$ has to be a non-positive function. We know that $\dot{I} =2\dot{q}\sin(q)$. Using Remark \ref{st: behaviour of p} together with the equalities
$\sin(q) = \frac{1}{\xi} + \overline{O}\left(\frac{1}{\xi^3}\right)$ (Taylor series of $\sin\bigl(\mathrm{arccot}(x)\bigr)$ at infinity is $\frac{1}{x}-\frac{1}{2x^3} + \overline{O}\left(\frac{1}{x^4}\right)$) and $\dot{q} =2p-m_1$, we can write 
\[
\dot{I} = \frac{4p-2 m_1}{\xi} + \overline{O}\left(\frac{1}{\xi^{\frac52}}\right),
\]
entailing that $\dot{I}\to 0$ on collision orbits by Remark \ref{st:remark about p}. Since $\sin(q)\ge0$, we can deduce that 
\begin{equation}\label{eq:critical}
\dot{q}\le0 \text{ in some left neighbourhood of }t^{\ast}. 
\end{equation}
Assume for a moment that $\dot{q}$ has a limit, as $t\to t^{\ast}$. Then from the equation for $\dot{I}$
\[
\lim\limits_{t\to t^{\ast}}\dot{q}  =\lim\limits_{t\to t^{\ast}}\frac{\dot{I}}{2\sin(q)}.
\]
Observe that in the expression above, the right hand side limit is of the type $\frac00$, and we can use L'H\^{o}pital's rule to estimate it:
\begin{equation}
    \label{eq: lim dot q}
\lim\limits_{t\to t^{\ast}}\dot{q} = \lim\limits_{t\to t^{\ast}}\frac{\ddot{I}}{2\cos(q)\dot{q}}
\end{equation}
 Now consider (\ref{eq: lim dot q}): if we assume that  $\dot{q}\to \dot{q}^*\in\overline{\mathbb{R}}$, that $\dot{q}^*$ can only be $\pm\infty$. Indeed,if $\dot{q}^*$ is a finite number, then the left hand side of  (\ref{eq: lim dot q}) is that number, and the right hand side is $\pm\infty$. Additionally, $\dot{q}$ can't tend to $+\infty$, since we know that it eventually becomes non-positive. Thus,  if $\dot{q}$ has a limit when $t\to t^{\ast}$, it must be $-\infty$.

Now, we need to prove that the limit exists in the first place.
\begin{figure}
    \centering
    \includegraphics[scale=.4]{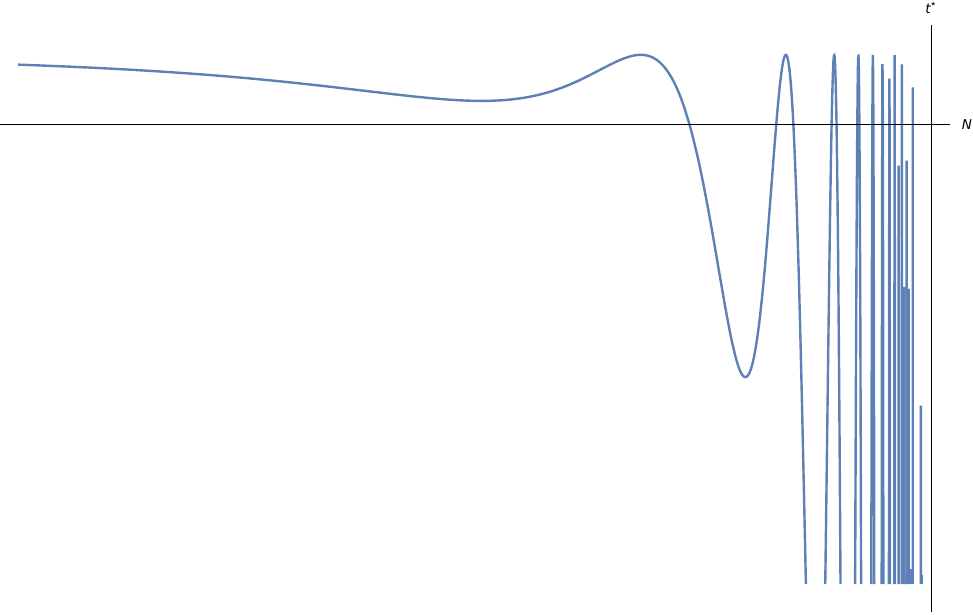}
    \caption{A rough sketch of the  assumed graph of $\dot{q}$, if it was to have no limit when $t\to t^{\ast}$. The vertical line is $t=t^{\ast}$,and the horizontal line is of the form $(N,t)$.}
    \label{fig: graph dot q}
\end{figure}
we do so in a number of steps. Assume that $\dot{q}$ \textbf{does not } have a limit as $t\to t^{\ast}$. Then the following claims hold:
\begin{enumerate}
    \item Zeros of $\ddot{q}$ accumulate to collision time.
    
    If we assume that there exists a left $\epsilon$-neighbourhood of $t^{\ast}$, such that $\ddot{q}$ does not have zeros in it, $\dot{q}$ will be a monotonic function in the neighbourhood in question and, as a consequence,  will have a limit, which, as stated above,  will  be $-\infty$ and will  contradict our current assumption that $\dot{q}$ does not have a limit. Therefore, there exists a sequence $t_n$, with $t_n\to t^{\ast}$ as $n\to\infty$, such that $\ddot{q}(t_n)=0$. 
    \item The values of $\dot{q}(t_n)\to-\infty$ when $n\to \infty$.
    
    Recall the equation for $\ddot{I}$:
    \begin{equation}
\label{eq: ddot I}
\ddot{I} = 2\cos(q)(\dot{q})^2 + 2\sin(q)\ddot{q}\to+\infty,
\end{equation}
by the proof of Lemma \ref{lemma:finite time}.
At $t_n$ the function  $\ddot{q}$ turns zero, and therefore $\left(\dot{q}(t_n)\right)^2 = \frac{\ddot{I}(t_n)}{\cos\left(q(t_n)\right)}$. Since $\ddot{I}(t_n)\to +\infty$ as $n\to +\infty$ (see Lemma \ref{lemma:finite time}) and $\cos\left(q(t_n)\right)\to 1$, we get that $\dot{q}(t_n)\to-\infty$, since $\dot{q}$ is a non-positive function in some neighbourhood of $t^{\ast}$.

    \item There exists a non-positive number $N$, such that the points with $\dot{q}\ge N$ accumulate to collision time.
    
    Since we have assumed that the limit for $t\to t^*$ of $\dot{q}(t)$ does not exist, in particular $\dot{q}(t)$ does not tend to $-\infty$. Therefore, there exists a (non-positive) number $N$, such that every left $\epsilon-$ neighbourhood of $t^{\ast}$ contains a point $\tilde{t}_{\epsilon}$, such that $\dot{q}(\tilde{t}_{\epsilon})\ge N$. Hence, there exists a sequence $\tilde{t}_n$, such that $\tilde{t}_n\to t^{\ast}$ and $\forall \ n \ \dot{q}(\tilde{t}_n)\ge N$.

\end{enumerate}
From the three observations above, we may conclude that the graph of $\dot{q}$ must look approximately as the one in Figure \ref{fig: graph dot q}: an oscillating function with ever decreasing minima, intersecting the horizontal line $N$ in every neighbourhood of $t^{\ast}$. 

Since both $t_n$ and $\tilde{t}_n$ converge to $t^{\ast}$, we can find  a subsequence  $\bar{t}_n$ of $t_n$, such that $\bar{t}_n\to t^{\ast}$,  $\ddot{q}(\bar{t}_n) = 0$ and $\dot{q}(\bar{t}_n)\ge N$ (effectively, we are taking a subsequence of local maxima of $\dot{q}$). At this point it is crucial to recall \eqref{eq:critical} that shows that $\dot{q}^2(\bar{t}_n)$ is bounded, since $N\leq \dot{q}(\bar{t}_n)\leq 0$. From here we get that on the points  $\bar{t}_n$ the right hand side  of (\ref{eq: ddot I}) is bounded, while the left hand side tends to $+\infty$, which leads to a contradiction. 

Therefore, $\dot{q}$ must have a limit with $t\to t^{\ast}$, and this limit can only be $-\infty$; moreover, since $m_1$ is a bounded function, $p$ must tend to $-\infty$ as well, since $\dot{q}= 2p-m_1$ (see the penultimate equation of the system (\ref{eq: two equal masses})) and $m_1$ is  a bounded function. 
\end{proof}

We can estimate how fast $p$ goes to $-\infty$:
\begin{lemma} $p\sim -\sqrt{\xi}$. \end{lemma}
\begin{proof}

In order to simplify our proof, we make a couple of observations:
\begin{enumerate}
    \item $\dot{q}\sim p$ and $p\sim\dot{q}$. This stems from the fact that $\dot{q} = 2p-m_1$, and therefore $\frac{\dot{q}}{p}\to2$ as $t\to t^{\ast}$ - these are exactly the conditions of entry 2 from Lemma \ref{st: lemma about the relations}.
    \item Similarly, $\ddot{I}\sim \xi$ and $\xi\sim\ddot{I}$ (as well as $\sqrt{\xi}\sim\sqrt{\ddot{I}}$ and $\sqrt{\ddot{I}}\sim\sqrt{\xi}$). We know from the proof of  Proposition \ref{st: behaviour of p} that $\ddot{I} = 4\xi + \overline{O}$. Hence, $\lim\limits_{t\to t^{\ast}}\frac{\ddot{I}}{\xi}= 4$, and the  statement immediately follows as in the instance above. 
    \item From the transitivity of $\preccurlyeq$ and Remark \ref{st: behaviour of p}, $\dot{q}\preccurlyeq-\sqrt{\xi}$ ( $|p|\preccurlyeq\sqrt{\xi}$ together with $p\to-\infty$ yield $p\preccurlyeq-\sqrt{\xi}$ and we know that $\dot{q}\sim p$ ).
    \item Using the transitivity of $\preccurlyeq$  together with entry 2, we get $\dot{q}\preccurlyeq-\sqrt{\ddot{I}}$.
    
    \item If we demonstrate that $\dot{q}\sim -\sqrt{\ddot{I}}$, the third entry of Lemma \ref{st: lemma about the relations} will allow us to conclude that $\dot{q}\sim-\sqrt{\xi}$ (the ratio of $\sqrt{\ddot{I}}$ and $\sqrt{\xi}$ tends to 2). 
    \item Third part of Lemma \ref{st: lemma about the relations} implies that if $\dot{q}\sim-\sqrt{\xi}$, then so is $p$ (as mentioned above, $\frac{p}{\dot{q}}\to\frac12$).
\end{enumerate}

This means that in order  to prove our initial claim, since we already know that $\dot{q} \preccurlyeq -\sqrt{\ddot{I}}$ it is enough to demonstrate that $\frac{\dot{q}}{\sqrt{\ddot{I}}}\not\to 0$.

 To this end, we observe that two separate cases can occur for $\ddot{q}$: it either has zeros accumulating to $t^{\ast}$, or there exists a value of $\epsilon>0$ such that $B_{\epsilon}(t^{\ast})$ does not have zeros of $\ddot{q}$ in it.

Assume that the latter is true; this makes $\dot{q}$ a monotonic function in some neighbourhood of $t^{\ast}$. Since we know that $\dot{q}\to-\infty$, $\ddot{q}$ must be negative. We consider the equation (\ref{eq: ddot I}) once again:

$$
\left(\dot{q}\right)^2 = \frac{\ddot{I} -2\sin(q)\ddot{q}}{2\cos(q)}.
$$
As we have established, $\ddot{q}$ is negative under the current assumptions and $\sin(q)$ is positive, and after division by $\ddot{I}$ of the expression above  we have
\[
\frac{\left(\dot{q}\right)^2}{\ddot{I}} = \frac{1}{2\cos(q)}  - \tan(q)\frac{\ddot{q}}{\ddot{I}}> \frac{1}{2}-\epsilon
\]
for some constant $\epsilon$ and in a sufficiently small neighbourhood of $t^{\ast}$; with this, $\frac{\dot{q}}{\sqrt{\ddot{I}}}\not\to 0$, and we are done.


Observe that the same logic as above applies when $\ddot{q}$ is just  non-positive.

Now consider the first option, and let zeros of $\ddot{q}$ accumulate to $t^{\ast}$ and, in the light of the last remark, and allow $\ddot{q}$ to take positive values in every neighbourhood of $t^{\ast}$. Then at the zeros of $\ddot{q}$ the following holds (we denote an arbitrary zero by $\tilde{t})$:
$$
4\xi + \overline{O} =\ddot{I}(\tilde{t}) = 2\cos(q(\tilde{t}))\left(\dot{q}(\tilde{t})\right)^2
$$
Therefore, at these points,  $\left(\dot{q}\right)^2$ coincides with the function $\frac{\ddot{I}}{2\cos(q)}$.

Suppose  that the zeros of $\ddot{q}$ form the sequence $\{t_n\}$, with $t_n\to t^{\ast}$ as $n\to \infty$; then on the elements of said sequence  $\frac{\left(\dot{q}\right)^2(t_n)}{\ddot{I}(t_n)} = \frac{1}{2\cos(t_n)}\to \frac12$, a sequence of points where the values of the ratio are separated from 0.  Thus,  in this case $\frac{\dot{q}}{\sqrt{\ddot{I}}}\not\to 0$ either, completing the  proof.
\end{proof}
\begin{lemma}
\label{st: estim on p}
If we represent $p$ as $p= g\sqrt{\xi}$, the function $g(t)$ is bounded away from 0 in some left $\epsilon$-neighbourhood of $t^{\ast}$. 
\end{lemma}
\begin{proof}

Firstly, we observe that $g(t)<0$: since $p$ is a negative function and $p\to-\infty$, there exists such a left neighbourhood of $t^{\ast}$ such   that $g$ doesn't have zeros in it.

The function $g$ can be represented as $g = \frac{p}{\sqrt{\xi}}$; both $p$ and $\xi$ are differentiable functions of time, and $\xi$ does not turn zero in some neighbourhood of $t^{\ast}$. Therefore, $g$ is a differentiable function in that same neighbourhood.

Now, let us assume the opposite of the statement of the lemma: i.e. that exists a sequence $\{t_n\}$, such that $g(t_n)\to 0$ and $t_n\to t^{\ast}$, as $n\to+\infty$. On the other hand, $g\not\to 0$ when $t\to t^{\ast}$; therefore, $g$ must assume a value less than some negative number $N$ in any neighbourhood of $t^{\ast}$.

We claim that these conditions entail existence of  a sequence of local maxima of $g(t)$, $\{\bar{t}_n\}$, such that $\dot{g}(\bar{t}_n)=0$, $\bar{t}_n\to t^{\ast}$ and $g(\bar{t}_n)\to 0$  when $n\to+\infty$. 

To see this, we observe that there exists a sequence $\{\tilde{t}_n\}$ interlaced with $\{t_n\}$ (that is, $t_{n-1}<\tilde{t}_n<t_n$ for all $n$), such that $g(\tilde{t}_n)<N$. Consider an interval $[\tilde{t}_n, \tilde{t}_{n+1}]$. On  both ends of  it the value of $g$ is less than $N$ and  inside it $g$ assumes a  value strictly bigger than $N$  -- at least at the point ${t}_n$. 
Therefore, $g$ must attain its maximum at some point $\bar{t}_n$  inside the interval in question. Consequently, $\dot{g}(\bar{t}_n)=0$ and $g(\bar{t}_n)\ge g(\tilde{t}_n)$. Taking such points $\bar{t}_n$ for all values of $n$, we construct the required sequence.

 Returning to our proof, we recall that since  $m_3\preccurlyeq \xi^{-\frac12}$, we can represent it as $m_3 = \gamma\xi^{-\frac12}$, with $\gamma$ a bounded function in some neighbourhood of $t^{\ast}$ (possibly tending to 0).

We substitute both the expressions for $m_3$ and $p$ into (\ref{eq:Ham and Cas common level set}), to get
\[
2\xi(g^2+\gamma^2-1) - 2\sqrt{\xi}(m_1g+m_2\gamma) = 2h-C.
\]

Evaluating the expression above at the points of the sequence  $\{\bar{t}_n\}$  allows us to conclude that $\gamma^2(\bar{t}_n)\to 1$, since the coefficient at $\xi$ must tend to 0 in order for the left hand side to be bounded ($m_1g + m_2\gamma$ is a bounded function). Evaluating the last equation of \eqref{eq: two equal masses with xi}, along the sequence $\{\bar{t}_n\}$ we get:
\begin{equation}\label{eq:contra3}
\dot{p}(\bar{t}_n) = -\left(\xi^2(\bar{t}_n)+1\right)\left(1+m_2(\bar{t}_n)m_3(\bar{t}_n)-2\gamma^2(\bar{t}_n)\right)>0,
\end{equation}
for sufficiently large $n$, since $1+m_2(\bar{t}_n) m_3(\bar{t}_n) - 2\gamma^2(\bar{t}_n))\to -1$, due to $m_2$ being a bounded function and $m_3$ tending to zero.

We can observe that on the other hand, differentiating $p=g\sqrt{\xi}$ gives:
\[
\dot{p} = \dot{g}\sqrt{\xi} + \frac{g}{2\sqrt{\xi}}\dot{\xi} = \dot{g}\sqrt{\xi}  -\frac{g}{2\xi}(\xi^2+1)(2p-m_1).
\]
Since $\dot{g}(\bar{t}_n)=0$ by our choice of $\{\bar{t}_n\}$, 
\[
\dot{p}(\bar{t}_n) =-\frac{g}{2\sqrt{\xi}}(\xi^2+1)(2p-m_1)<0,
\]
seeing as  $g<0$ and $p\to-\infty$, but this contradicts \eqref{eq:contra3}. Therefore, $g$ has to be bounded away from 0. 
\end{proof}

Armed with this,  we can estimate the behaviour of $m_3$. Firstly, we prove the following crucial 
\begin{lemma}\label{lemma:xiastime}
The function $\xi$ is a strictly monotonous function of time in some neighbourhood of $t^{\ast}$. 
\end{lemma}
\begin{proof}
Indeed, $\dot{q}\to-\infty$ and  becomes strictly negative at some point. But $\dot{\xi} = -(\xi^2 + 1)\dot{q}$ and is therefore strictly greater than 0. 
\end{proof}

\begin{proposition}
We have
\begin{equation}
m_3\preccurlyeq\frac{1}{\xi}.
\end{equation}
\end{proposition}
\begin{proof}
Because of Lemma \ref{lemma:xiastime} we can consider $m_3$ as a function of $\xi$ instead of $t$ in a suitable left neighbourhood of $t^*$. For values $\xi_1$ and $\xi_0$ corresponding to this suitable left neighbourhood of $t^*$ we have 
\begin{equation}
    \begin{split}
        \label{eq: m3}
   \left| m_3(\xi_1)-m_3(\xi_0)\right|&=\left|\int\limits_{\xi_0}^{\xi_1}\frac{\dot{m_3}}{\dot{\xi}}\mathrm{d}\xi \right|\le\int\limits_{\xi_0}^{\xi_1}\left|-\frac{m_2p -m_1m_3\xi}{(\xi^2+1)(2p-m_1)}\right|\mathrm{d}\xi
  =\int\limits_{\xi_0}^{\xi_1}\left|\frac{p\left(m_2 -f(\xi)\right)}{(\xi^2+1)p(2-\tilde{f}(\xi))}\right|\mathrm{d}\xi =\\&   =\int\limits_{\xi_0}^{\xi_1}\left|\frac{m_2 -f(\xi)}{(\xi^2+1)(2-\tilde{f}(\xi))}\right|\mathrm{d}\xi = \int\limits_{\xi_0}^{\xi_1}\left|\frac{m_2 -f(\xi)}{\xi^2(1 + \frac{1}{\xi^2})(2-\tilde{f}(\xi))}\right|\mathrm{d}\xi  .
    \end{split}
\end{equation}
In the expression above, $\tilde{f}(\xi) = \frac{m_1}{p} \to 0$, since $p\to-\infty$. As for the function $f(\xi)$, it's explicit form is $f(\xi) = \frac{m_1m_3\xi}{p}$. Using the terminology of Lemma \ref{st: estim on p}, $f(\xi) = \frac{m_1\gamma\sqrt{\xi}}{g \sqrt{\xi}} = \frac{m_1\gamma}{g}$. From that same Lemma, we know that $g$ is bounded away from 0 in some neighbourhood of $t^{\ast}$, and consequently it is bounded away from zero as a function of $\xi$ in a neighbourhood of $+\infty,$ as well. Therefore in that neighbourhood the function $f$ is bounded. Then, (\ref{eq: m3}) can be further rewritten as 
\begin{equation}
    \label{eq:estim on m3 but better}
    \begin{split}
     \int\limits_{\xi_0}^{\xi_1}\left|\frac{m_2 -f(\xi)}{\xi^2(1 + \frac{1}{\xi^2})(2-\tilde{f}(\xi))}\right|\mathrm{d}\xi\le \int\limits_{\xi_0}^{\xi_1}\left|\frac{m_2 -f(\xi)}{\xi^2(2-\tilde{f}(\xi))}\right|\mathrm{d}\xi   \le \int\limits_{\xi_0}^{\xi_1}\frac{C_0}{\xi^2}\mathrm{d}\xi = C_0\left(\frac{1}{\xi_0} - \frac{1}{\xi_1}\right)
    \end{split}
\end{equation}
for some non-negative constant $C_0$.

Now we combine the leftmost part of (\ref{eq: m3}) with the rightmost one of (\ref{eq:estim on m3 but better}) and take the limit $\xi_1\to +\infty$. In this case, both $\frac{1}{\xi_1}$ and $m_3(\xi_1)$  turn 0, and we are left with 
\[
\left|m_3(\xi_0)\right|\le \frac{C_0}{\xi_0},
\]

meaning that $m_3\preccurlyeq\frac{1}{\xi}$.
\end{proof}


This important fact tells us that $m_3\xi$ is, in fact, a bounded function, allowing us to  describe the behaviour of $p$ even better: 
\begin{lemma} For $t$ in a suitable left neighbourhood of $t^*$ we have:
\begin{equation}
\label{eq: estim on p}
p  = -\sqrt{\xi}\left(1 - \frac{m_1}{2\sqrt{\xi}} + \overline{O}\left(\frac{1}{\xi}\right)\right).
\end{equation}
\end{lemma}
\begin{proof}
 Using the newly acquired  estimations on $m_3$, we can rewrite (\ref{eq: estimate p}) as
\[
2\left(p-\frac{m_1}{2} \right)^2- 2\xi = 2a(t),
\]
where $2a(t) := 2h-C + \frac{m_1^2 + m_2^2}{2} - \left(\sqrt{2}m_3\xi - \frac{m_2}{\sqrt{2}}\right)^2$, a bounded function owing to $m_3\preccurlyeq\frac{1}{\xi}$.

Therefore, 
\[
\left(p - \frac{m_1}{2}\right)^2 = \xi+ a(t).
\]
 Then we obtain (since we already know that $p\to -\infty$ at $t\to t^*$)
\begin{equation*}
    \begin{split}
p = &\frac{m_1}{2} - \sqrt{\xi +a(t)} = \frac{m_1}{2} -\sqrt{\xi}\sqrt{1 + \frac{a(t)}{\xi}} = \frac{m_1}{2} -\sqrt{\xi}\left(1 + \frac{a(t)}{2\xi} - \frac{a^2(t)}{8\xi^2} + \overline{O}\left(\frac{1}{\xi^3}\right)\right) =\\=& \frac{m_1}{2} -\sqrt{\xi}\left(1 + \overline{O}\left(\frac{1}{\xi}\right)\right),
\end{split}
\end{equation*}
using boundedness of $a$ and the fact that the Taylor series of $\sqrt{1+x} $ at $x=0$ is $1 + \frac{x}{2} -\frac{x^2}{8} + \overline{O}\left(\frac{1}{x^3}\right)$. This last expression can clearly be rewritten as  \eqref{eq: estim on p}. 
\end{proof}

\begin{theorem}
\label{st: theorem omega limit set}
Along any solution with initial data $D$ leading to a collision, the length of the projection of the solution on the Casimir sphere $\{m_1^2+m_2^2+m_3^2=C\}$ is finite. Thus the omega-limit set $\Omega(D)$ is a singleton of the type described above, $\Omega(D)=\{m_1=m_1^*, m_2=m_2^*, m_3=0, q=0, p=-\infty\}.$
\end{theorem}
\begin{proof}
 First we need some more estimates involving the $m_i$s variables. 
Our familiar estimate on $m_3$ (after sending the upper integration limit to $+\infty$) turns into 
\begin{equation}
\label{eq: m_3 explicit}
m_3(\xi_0) = \int\limits_{\xi_0}^{+\infty}\frac{m_2p - m_1m_3\xi}{(\xi^2 + 1)(2p-m_1)}\mathrm{d}\xi.
\end{equation}
In light of the new estimation \eqref{eq: estim on p}, we have
\begin{equation}
\label{eq: numerator}
    \begin{split}
        m_2p-m_1m_3\xi &= -m_2\sqrt{\xi}\left(1 -\frac{m_1}{2\sqrt{\xi}} + \overline{O}\left(\frac{1}{\xi}\right)\right) - m_1m_3\xi =\\&= -\sqrt{\xi}\left(m_2 -\frac{m_1m_2}{2\sqrt{\xi}} + \overline{O}\left(\frac{1}{\xi}\right) + m_1m_3\sqrt{\xi}\right).
    \end{split}
\end{equation}
We denote 
\begin{equation}
    \label{eq: f(t)}
    f(\xi) :=\frac{m_1m_2}{2\sqrt{\xi}} - m_1m_3\sqrt{\xi} - \overline{O}\left(\frac{1}{\xi}\right).
    \end{equation}
    Clearly, $f(\xi)\to 0 $ as $\xi\to+\infty$.
    
    \noindent    Additionally, 
    \begin{equation}
        \label{eq: 2p-m1}
        2p-m_1 = -\sqrt{\xi}\left(2+\overline{O}\left(\frac{1}{\xi}\right)\right).
    \end{equation}
 Using (\ref{eq: 2p-m1}) and (\ref{eq: numerator}), we can rewrite (\ref{eq: m_3 explicit}) as
\begin{equation}
\label{eq: m3 another}
m_3(\xi_0) = \int\limits_{\xi_0}^{+\infty}\frac{-\sqrt{\xi}(m_2 -f(\xi))}{-\xi^2\sqrt{\xi}\left(1 + \frac{1}{\xi^2}\right)\left(2   + \overline{O}\left(\frac{1}{\xi}\right)\right)}\mathrm{d}\xi = \int\limits_{\xi_0}^{+\infty}\frac{m_2 - f(\xi)}{\xi^2\left(1 + \frac{1}{\xi^2}\right)\left(2  +\overline{O}\left(\frac{1}{\xi}\right)\right)}\mathrm{d}\xi.
\end{equation}

\noindent Consider 
\begin{equation}
\label{eq: F(xi)}
F(\xi) := \frac{m_2 - f(\xi)}{\left(1 + \frac{1}{\xi^2}\right)\left(2   + \overline{O}\left(\frac{1}{\xi}\right)\right)},
\end{equation}
as we have remarked above, a bounded function; it is clear that  as $\xi\to+\infty$,  $F(\xi)\to\frac{m_2}{2}$. Let us estimate how fast this happens; to that end, take the common denominator in the following expression:
\[
\frac{m_2 - f(\xi)}{\left(1 + \frac{1}{\xi^2}\right)\left(2   + \overline{O}\left(\frac{1}{\xi}\right)\right)}- \frac{m_2}{2} = \frac{-f(\xi) - \frac{m_2}{2}\overline{O}\left(\frac{1}{\xi}\right) - \frac{m_2}{\xi^2} -\frac{m_2}{2}\overline{O}\left(\frac{1}{\xi^3}\right)}{\left(1 + \frac{1}{\xi^2}\right)\left(2   + \overline{O}\left(\frac{1}{\xi}\right)\right)}.
\]
Recall that from (\ref{eq: f(t)}), $f(\xi) = \frac{m_1m_2}{2\sqrt{\xi}} - m_1m_3\sqrt{\xi} - \overline{O}\left(\frac{1}{\xi}\right)\preccurlyeq \frac{1}{\sqrt{\xi}}$ (since $m_3\preccurlyeq\frac{1}{\xi}$) and thus the speed  with which $F(\xi)$ tends to $\frac{m_2}{2}$ is at least $\frac{1}{\sqrt{\xi}}$, entailing that 
\begin{equation}
    \label{eq: F(xi) expand}
F(\xi) = \frac{m_2}{2} + \frac{G(\xi)}{\sqrt{\xi}},
\end{equation}
where $G(\xi)$ is some bounded function. 

Lastly, we estimate the behaviour of $m_3 - \frac{m_2}{2\xi}$. Using the definition of $F(\xi)$ from (\ref{eq: F(xi)}) and the last equality from (\ref{eq: m3 another}) together with (\ref{eq: F(xi) expand}), we can write 
\[
m_3(\xi_0) = \int\limits_{\xi_0}^{+\infty}\frac{\frac{m_2}{2} + \frac{G(\xi)}{\sqrt{\xi}}}{\xi^2}\mathrm{d}\xi.
\]

In order to save space, we quickly remark that $\lim\limits_{t\to t^{\ast}}\left(m_3-\frac{m_2}{2\xi}\right) = 0$ and apply exactly the same methodology as in (\ref{eq: m3}), (\ref{eq:estim on m3 but better}) and (\ref{eq: m_3 explicit}) to get 
\begin{equation}
\label{eq:another bloody estimation}
   \begin{split}
  \left(  m_3 - \frac{m_2}{2\xi}\right)(\xi_0) = \int\limits^{+\infty}_{\xi_0}\frac{\frac{m_2}{2} + \frac{G(\xi)}{\sqrt{\xi}}}{\xi^2}\mathrm{d}\xi + \int\limits_{\xi_0}^{+\infty}\frac{1}{\dot{\xi}}\dot{\left(\frac{m_2}{2\xi}\right)}\mathrm{d}\xi,
   \end{split} 
\end{equation}
since 
\[ 
\frac{m_2}{2\xi}(\xi_1)- \frac{m_2}{2\xi}(\xi_0) = \int\limits_{\xi_0}^{\xi_1}\frac{1}{\dot{\xi}}\dot{\left(\frac{m_2}{2\xi}\right)}\mathrm{d}\xi
\]
and $\lim\limits_{\xi\to+\infty}\frac{m_2}{2\xi}=0$. 
The second part of (\ref{eq:another bloody estimation}) is given by
\begin{equation*}
\begin{split}
\dot{\left(\frac{m_2}{2\xi}\right)} &= \frac{\dot{m}_2}{2\xi} - \frac{m_2}{2\xi^2}\dot{\xi} = \frac{m_1m_2\xi -m_3p - 2m_1m_3\xi^2}{2\xi} - \frac{m_2}{2\xi^2}\dot{\xi}.
\end{split}
\end{equation*}
Therefore, we have for the absolute value of (\ref{eq:another bloody estimation}): \begin{equation*}
    \begin{split}
      \left| (\ref{eq:another bloody estimation})\right|  &=\left| \int\limits_{\xi_0}^{+\infty}\frac{m_2}{2\xi^2}  + \frac{G(\xi)}{\xi^{\frac52}}- \frac{m_1m_2\xi -m_3p - 2m_1m_3\xi^2}{2\xi(\xi^2+1)(2p-m_1)}-\frac{m_2}{2\xi^2}\mathrm{d}\xi\right|\le\\&\le \int\limits_{\xi_0}^{+\infty}\left|\frac{m_2}{2\xi^2}  + \frac{G(\xi)}{\xi^{\frac52}}- \frac{m_1m_2\xi -m_3p - 2m_1m_3\xi^2}{2\xi(\xi^2+1)(2p-m_1)}-\frac{m_2}{2\xi^2}\right|\mathrm{d}\xi=\\
      &=\int\limits_{\xi_0}^{+\infty} \left|\frac{G(\xi)}{\xi^{\frac52}}- \frac{m_1m_2\xi -m_3p - 2m_1m_3\xi^2}{2\xi(\xi^2+1)(2p-m_1)}\right|\mathrm{d}\xi\le
        \int\limits_{\xi_0}^{+\infty}\frac{C_1}{\xi^{\frac52}}\mathrm{d}\xi = C'_1\xi_0^{-\frac32}
    \end{split}
\end{equation*}
for some non-negative constant $C'_1$.

Therefore, we have demonstrated that  $2m_3\xi - m_2\preccurlyeq\xi^{-\frac12}$.

In order to complete our investigation, we recall that $m_1,\  m_2, \ m_3 $ 'live' on the invariant two-dimensional sphere of radius $\sqrt{C}$, where $C$ is the value of the Casimir function. Every trajectory of the system can be projected onto this sphere, and that is what we do next. For this new trajectory on $S^2$, we want to demonstrate that its $\omega$-limit set in $m_1, \ m_2, \ m_3$ is a point. To do so, we compute its length: if the length in question is finite, the omega-limit set is a singleton. 

Denoting with $L(\Gamma)$ the length of the collision solution from time $t_0$ to the time of collision $t^*$ projected on the Casimir sphere and once again integrating with respect to $\xi$ we have:
\begin{footnotesize}
\begin{equation}
\label{eq:last}
    \begin{split}
        &L(\Gamma)=\int\limits_{t_0}^{t^{\ast}}\sqrt{\dot{m}_1^2 + \dot{m}_2^2 + \dot{m}_3^2}\mathrm{d}t =\\& \int\limits_{\xi_0}^{+\infty}\frac{\sqrt{(m_1^2 + m_2^2)\xi^2(m_2 - 2m_3\xi)^2 + (m_3p)^2 + m_3^4\xi^2 -2 m_3^2\xi^2m_2(m_2 - 2m_3\xi) -2m_3pm_1\xi(m_2 - 2m_3\xi) + (m_2p-m_1m_3\xi)^2} }{(\xi^2 + 1)(m_1-2p)}\mathrm{d}\xi,
    \end{split}
\end{equation}
\end{footnotesize}

Let us expound on the last expression by going through the summands under the square root, using all of our previous estimations. 

We know that $2m_3\xi - m_2\preccurlyeq \xi^{-\frac12}$; therefore, $(m_1^2 +m_2^2)\xi^2(m_2-2m_3\xi)^2\preccurlyeq \xi$, and for the same reason $2m_3^2\xi^2m_2(m_2-2m_3\xi)\to 0$ and  $2m_3pm_1\xi(m_2-2m_3\xi)$ is bounded. Additionally, $m_3^2p^2\to0$, $m_3^4\xi^2\to 0$ and $(m_2p - m_1m_3\xi)^2\preccurlyeq \xi$. Hence, the growth of the terms under the square root does not exceed that of $\xi$; taking the square root reduces it to $\xi^\frac12$, which cancels out with that of $p$, leaving $\xi^2 + 1$ in the denominator.  We may conclude that  

\[
(\ref{eq:last})\le\int\limits_{\xi_0}^{+\infty}\frac{K(\xi)}{\xi^2}\mathrm{d}\xi,
\]

with $K(\xi)$ some bounded  positive function.

Since the length of the projected solution is finite, the projection onto the sphere of the omega-limit set is a singleton. But by definition of collision solutions and what we proved before the other dynamical variables $(q, p)$ tend to $(0,-\infty)$, so the omega-limit set is of the form stated. 
\end{proof}

\subsection{Behaviour as a pair}

We have heretofore been describing the system in terms of the reduced variables; the global dynamics of the system are still unclear. However, the nature of the reduction allows us to make a number of observations. 

\begin{observation}
The two bodies move towards the equator that is in the plane perpendicular to the vector of the angular momentum in the fixed frame. 
\end{observation}
Recall that $\mathbf{m} = \left(m_1,m_2,m_3\right)$ represent the angular momentum in the body frame. Therefore, the angular momentum $\mathbf{L}$ in the fixed frame is given by $\mathbf{L} = g \mathbf{m}$, where $g$ is as in (\ref{eq:Euler angles}). Rotating the fixed frame accordingly, we may assume that $\mathbf{L} = (0,0,L)$ and, consequently,
\begin{equation}
    \label{eq: m}
\mathbf{m} = \left(L\sin\theta\sin\phi ,\ L \sin\theta\cos\phi  ,\ L \cos\theta \right)
\end{equation}
Since $m_3\to0$, it immediately follows that $\theta\to\frac{\pi}{2}$. Recall that (Figure \ref{fig: two coord systems}) $\pi-\theta$ is the angle between the vertical fixed axis and the coordinate vector of the first body: this  
forces the first body into the  plane perpendicular to the vector of angular momentum. Since $q\to0$, the second body will be moving towards the equator as well. 
\begin{observation}
The two bodies cannot rotate around each other infinitely many times. 
\end{observation}

Now, consider the angle $\phi$ in Figure \ref{fig: two coord systems}: it is the angle between the $x'$-axis and the vector $x''$ along the line where the $x-y$ and $x'-y'$ planes intersect. Since the coordinate vector of the first body coincides with $z'$ and the second body always lies in $z'-y'$ plane,  this angle describes the relative rotation of the two bodies. 

Since we know that $\theta \rightarrow \frac{\pi}{2}$, we have that $\sin(\theta) \rightarrow 1$. Moreover, from (\ref{eq: m}), $m_1 = L\sin\theta\sin\phi ,\  m_2= L\sin\theta\cos\phi  $. Therefore, if the two bodies keep rotating around each other, the angle  $\phi(t)$ must keep changing; on the other hand, we know by the description of the omega-limit set that $m_1\rightarrow m_1^*$ and $m_2\rightarrow m_2^*$ in finite time, so therefore the two bodies can not rotate around each other infinitely many times.
\begin{figure}
    \centering
    \subfigure[In $q-p$ coordinates]{ \includegraphics[scale =0.45]{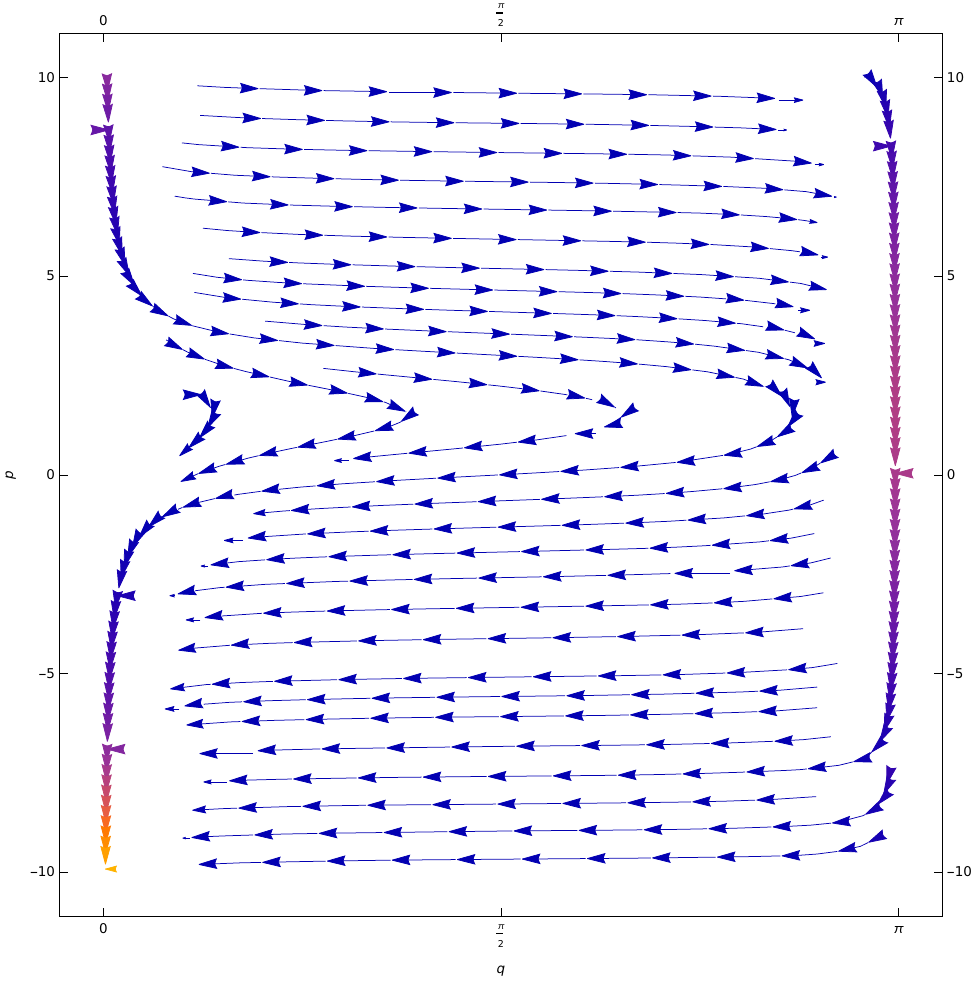}}
  \subfigure[In $\xi-p$ coordinates]{ \includegraphics[scale =0.45]{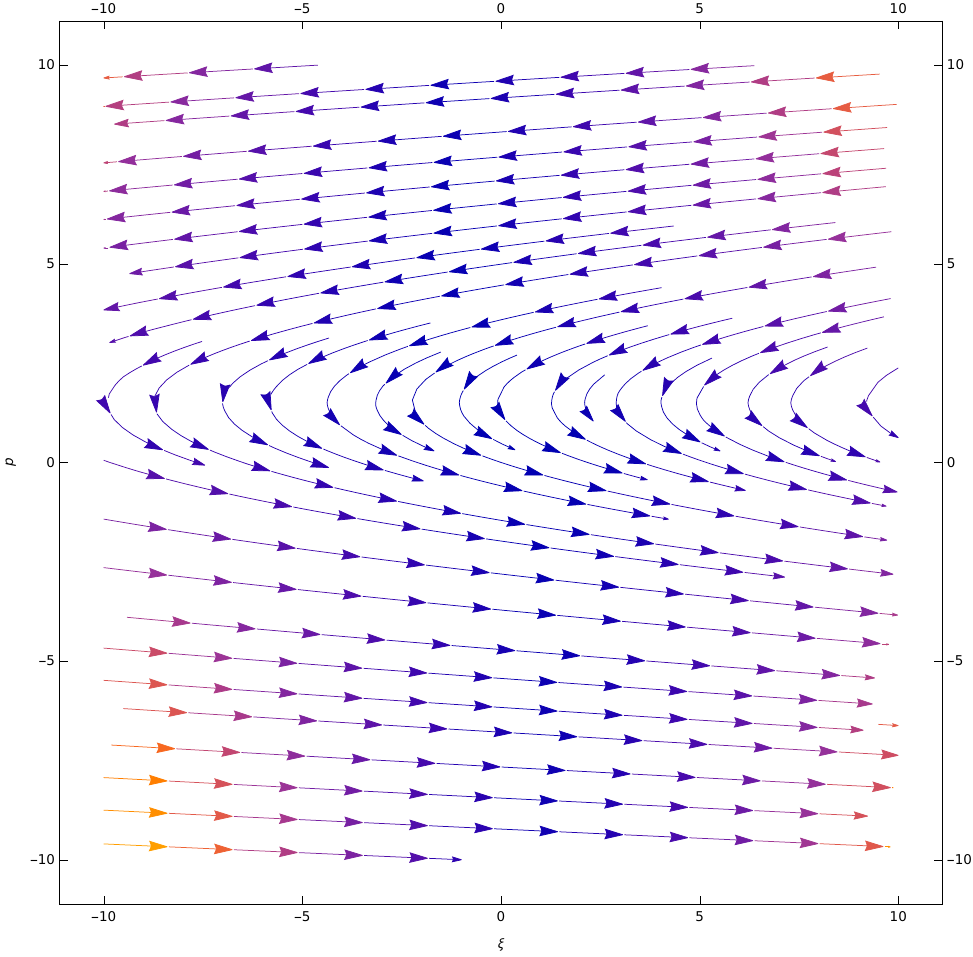}}
\caption{Dynamics in the invariant plane with $C=9$}
\label{fig:Dynamics in the invariant plane}
\end{figure}
\
\section{Regularisation}
In this Section we study the regularisation of collision singularities for the system \eqref{eq: two equal masses with xi}. This entails the following steps. First we want the collision singularity (the omega-limit singleton $\Omega(D)$) to become an ordinary equilibrium. This is achieved via coordinates transformation and possibly a time re-scaling. If the resulting equilibrium is degenerate (which is the case), then we proceed with blow-up in order to understand the dynamics of near collision trajectories. In the general case, we will have to perform a high-dimensional non-homogeneous blow up, using two different charts.

For the sake of being self-contained, we introduce briefly the blow-up procedure. Let $X$ be a vector field in an open neighborhood $U$ of zero in $\mathbb{R}^n.$ Suppose zero is an isolated degenerate equilibrium of $X$, meaning that the determinant of the linearization of $X$ at zero is zero (in our case the linearization is identically zero). In order to understand the dynamics near zero, one introduces the map $$\phi: M:=S^{n-1}\times [0, +\infty)\rightarrow \mathbb{R}^n,$$
$$\phi: ((q_1, \dots q_{n-1}),r)\mapsto (x_1=r^{\alpha_1}\cos(q_1), x_2=r^{\alpha_2}\sin(q_1)\cos(q_2), \dots, x_n=r^{\alpha_n}\sin(q_1)\dots \sin(q_{n-1})),$$
where $q_1, \dots q_n$ are angular coordinates on $S^{n-1}$, and $\{\alpha_1, \dots, \alpha_n\}$ are positive integers. Such a map is a called a blow-up with weights $\{\alpha_1, \dots, \alpha_n\}$. It is in particular called homogeneous if all the weights are equal, non-homogeneous otherwise. Notice that $M$ is a semi-infinite cylinder over $S^{n-1}$ with boundary $\partial M=S^{n-1}\times\{0\}.$ Moreover, $\phi: M\setminus\partial M \rightarrow \mathbb{R}^n\setminus{\{0\}}$ is a global diffeomorphism (the change of coordinates given by generalized spherical coordinates). Notice that the pre-image via $\phi$ of the origin in $\mathbb{R}^n$ is $\partial M=S^{n-1}\times \{0\}$.
$\partial M$ is sometimes called {\em the exceptional divisor}, a terminology originating in Algebraic Geometry. 

One can define the vector field $\hat{X}$ on $M$ such that $\phi_{*}(\hat{X})=X,$ where $\phi_{*}$ is the push-forward map. $\hat{X}$ is called the {\em complete lift} of $X$ to the blow-up.  Computationally this simply means to re-write the vector field $X$ in coordinates $((q_1, \dots q_n), r)$ for $r>0$ and then considering its extension to the boundary of $M$ too, i.e. for $r=0.$ 
Let's denote with $j_k(X)(p)$ the $k$-th jet of $X$ at $p$. We can think of the $k$-th jet as the collection of all partial derivatives of $X$ up to order $k$ at the point $p$. If $j_k(X)(0)=0$, then $j_k(\hat{X})(u)=0$ for all $u\in \partial M.$
In particular, $\hat{X}$ is also very degenerate on $\partial M.$

In order to remove the degeneracy, we introduce the vector field $\bar{X}=\frac{1}{r^k}\hat{X}$, where $k$ is a positive integer. Concretely, $k$ can be chosen to be the smallest positive integer such that $\bar{X}_{|\partial M}$ is not identically zero. Then $\bar{X}$ is also a smooth vector field on $M$. In particular, on $M\setminus{\partial M}$, division by $r^k$ does not change the orbits of $\hat{X}$ or their orientation, but only the time parametrization. In general, the equilibria of $\bar{X}_{|\partial M}$ will be isolated and often non-degenerate (if they are one needs to perform more blow-ups on those that are degenerate).

Thus, one equilibrium at the origin is replaced by multiple equilibria on the blow-up sphere. If they all are non-degenerate, using standard methods, one can determine the type of new equilibria and understand how the vector field behaves around the exceptional divisor: this is akin to zooming in very closely at the origin. Once all new equilibria are described, we can think of collapsing the blow-up sphere back to a point, to get a picture of the behavior of the vector field $X$ near $0$. 
For more on this (in the planar case) see for instance \cite{dumortier2006qualitative}. 

\subsection{Invariant plane}

Prefacing the more general case, we investigate in depth a particular instance of our problem: namely, the one where two bodies are restricted to move on a great circle.  

 As it was mentioned above, due to the rotational symmetry of the system we may chose the reference frame in such a way that the angular momentum is along the $z$-axis and the two bodies move on the circle in the $(x,y)$-plane. The  angle $\theta$ between the $z$-axis of the fixed frame and the $z'$-axis of the moving frame is $\frac{\pi}{2}$, meaning that $m_3 = 0$ and $\dot{m}_3  = m_2p-m_1m_3\xi= 0$. 
 
 Since $\dot{p} = -\csc^2(q)$ in this case, $p$ can't be identically zero, and therefore $m_2 =0$, forcing $m_1 = \pm\sqrt{C}$. 
 
 Hence, in this case the system is reduced to a two-dimensional one in $(q,p)$ (or $(\xi,p)$) invariant plane.
 
 As a result, we are left with the following system:
 \begin{equation}
     \label{eq:reduced tw dim}
     \begin{cases}
      \dot{q} = 2p  \pm\sqrt{C},\\
      \dot{p}  =-\csc^2(q);
     \end{cases} \ \mathrm{or} \ \begin{cases}
      \dot{\xi} = -(2p  \pm\sqrt{C})(\xi^2 + 1);\\
      \dot{p}  =-(\xi^2 + 1).
     \end{cases} 
 \end{equation}

Phase portraits for both of the systems are depicted in Figure \ref{fig:Dynamics in the invariant plane}. 

\begin{rmk}
It can be checked analytically (via the explicit formulae for $\psi$, for example) that the collision of the two bodies will always happen head on in this case. 
\end{rmk}

\begin{figure}
\centering
\begin{tikzpicture}[x=0.75pt,y=0.75pt,yscale=-0.7,xscale=0.7]

\draw   (170,210.84) .. controls (170,143.92) and (224.25,89.67) .. (291.16,89.67) .. controls (358.08,89.67) and (412.33,143.92) .. (412.33,210.84) .. controls (412.33,277.75) and (358.08,332) .. (291.16,332) .. controls (224.25,332) and (170,277.75) .. (170,210.84) -- cycle ;
\draw    (109.11,209.67) -- (235.11,210.67) ;
\draw   (141.56,202.85) .. controls (145.37,206.92) and (149.19,209.36) .. (153,210.17) .. controls (149.19,210.99) and (145.37,213.43) .. (141.56,217.5) ;
\draw    (344.11,209.67) -- (470.11,210.67) ;
\draw  [fill={rgb, 255:red, 0; green, 0; blue, 0 }  ,fill opacity=1 ] (167.59,209.14) .. controls (167.87,207.89) and (169.12,207.11) .. (170.36,207.4) .. controls (171.61,207.68) and (172.39,208.92) .. (172.11,210.17) .. controls (171.82,211.42) and (170.58,212.2) .. (169.33,211.91) .. controls (168.08,211.63) and (167.3,210.39) .. (167.59,209.14) -- cycle ;
\draw  [fill={rgb, 255:red, 0; green, 0; blue, 0 }  ,fill opacity=1 ] (368.92,119.81) .. controls (369.21,118.56) and (370.45,117.78) .. (371.7,118.06) .. controls (372.95,118.35) and (373.73,119.59) .. (373.44,120.84) .. controls (373.16,122.08) and (371.91,122.87) .. (370.67,122.58) .. controls (369.42,122.3) and (368.64,121.05) .. (368.92,119.81) -- cycle ;
\draw  [fill={rgb, 255:red, 0; green, 0; blue, 0 }  ,fill opacity=1 ] (207.26,120.14) .. controls (207.54,118.89) and (208.78,118.11) .. (210.03,118.4) .. controls (211.28,118.68) and (212.06,119.92) .. (211.77,121.17) .. controls (211.49,122.42) and (210.25,123.2) .. (209,122.91) .. controls (207.75,122.63) and (206.97,121.39) .. (207.26,120.14) -- cycle ;
\draw  [fill={rgb, 255:red, 0; green, 0; blue, 0 }  ,fill opacity=1 ] (410.07,210.32) .. controls (410.35,209.07) and (411.6,208.29) .. (412.85,208.58) .. controls (414.09,208.86) and (414.87,210.1) .. (414.59,211.35) .. controls (414.3,212.6) and (413.06,213.38) .. (411.81,213.09) .. controls (410.57,212.81) and (409.79,211.57) .. (410.07,210.32) -- cycle ;
\draw   (196,217.5) .. controls (193.02,213.57) and (190.04,211.21) .. (187.06,210.42) .. controls (190.04,209.64) and (193.02,207.28) .. (196,203.35) ;
\draw   (179.32,230.9) .. controls (175.9,234.46) and (174.04,237.77) .. (173.73,240.84) .. controls (172.48,238.02) and (169.69,235.45) .. (165.34,233.11) ;
\draw   (164.85,184.7) .. controls (169.57,183.25) and (172.8,181.25) .. (174.56,178.72) .. controls (174.28,181.79) and (175.48,185.4) .. (178.15,189.55) ;
\draw    (165.52,76.19) -- (253.99,166.08) ;
\draw    (416.7,75.9) -- (326.81,164.38) ;
\draw   (413.65,176.3) .. controls (410.97,180.45) and (409.77,184.06) .. (410.06,187.13) .. controls (408.3,184.59) and (405.06,182.6) .. (400.35,181.14) ;
\draw   (402.09,242.76) .. controls (406.27,240.13) and (408.88,237.37) .. (409.92,234.47) .. controls (410.44,237.51) and (412.53,240.68) .. (416.19,243.99) ;
\draw   (393.5,217) .. controls (390.52,213.07) and (387.54,210.71) .. (384.56,209.92) .. controls (387.54,209.14) and (390.52,206.78) .. (393.5,202.85) ;
\draw   (425.06,202.85) .. controls (428.87,206.92) and (432.69,209.36) .. (436.5,210.17) .. controls (432.69,210.99) and (428.87,213.43) .. (425.06,217.5) ;
\draw   (195.87,97.26) .. controls (195.2,102.15) and (195.64,105.92) .. (197.19,108.59) .. controls (194.53,107.03) and (190.75,106.6) .. (185.86,107.27) ;
\draw   (223.69,145.09) .. controls (224.36,140.2) and (223.92,136.42) .. (222.37,133.76) .. controls (225.03,135.31) and (228.81,135.75) .. (233.69,135.08) ;
\draw   (185.66,139.26) .. controls (190.56,138.64) and (194.1,137.24) .. (196.26,135.05) .. controls (195.46,138.02) and (196.01,141.79) .. (197.92,146.33) ;
\draw   (235.44,111.82) .. controls (230.89,109.9) and (227.13,109.35) .. (224.16,110.16) .. controls (226.34,107.99) and (227.74,104.45) .. (228.36,99.56) ;
\draw   (361.19,140.27) .. controls (356.3,139.59) and (352.53,140.03) .. (349.87,141.59) .. controls (351.42,138.92) and (351.86,135.15) .. (351.19,130.26) ;
\draw   (379.36,102.58) .. controls (384.25,103.25) and (388.03,102.81) .. (390.69,101.26) .. controls (389.14,103.92) and (388.7,107.7) .. (389.37,112.59) ;
\draw   (392.42,131.01) .. controls (390.51,135.56) and (389.95,139.32) .. (390.76,142.3) .. controls (388.59,140.11) and (385.06,138.71) .. (380.16,138.09) ;
\draw   (354.1,115.37) .. controls (352.58,109.95) and (350.26,105.88) .. (347.16,103.19) .. controls (351.04,104.52) and (355.72,104.5) .. (361.18,103.11) ;

\end{tikzpicture}
\caption{blow-up of the point $\eta=\zeta=0$ in the invariant plane}
\label{fig: blow-up in the invatiant plane}
\end{figure}
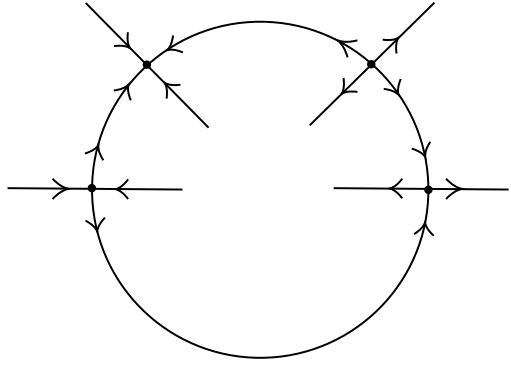
In order to regularise the system, we introduce a coordinate transformation $\eta = \frac{1}{\xi}$ and $\zeta = \frac{p}{\xi}$. It follows from Remark \ref{st:remark about p}  that $\eta,\ \zeta\to0$. After the change, (\ref{eq:reduced tw dim}) becomes
\begin{equation*}
    \begin{cases}
    \dot{\eta} = \frac{1}{\eta}(2\zeta \pm\sqrt{C}\eta)(\eta^2 + 1)\\
    \dot{\zeta} = \frac{1}{\eta^2}(\eta^2 + 1)(2\zeta^2 \pm\sqrt{C}\eta\zeta -\eta).
    \end{cases}
\end{equation*}
To eliminate singularities at the origin, we introduce fictitious time $\tau$, such that $\frac{\mathrm{d}t}{\mathrm{d}\tau} = \eta^2$. If we denote differentiating with respect to $\tau$ as $'$, we get
\begin{equation}
    \label{eq: regularisation in the invariant plane}
    \begin{cases}
   \eta' = \eta(2\zeta \pm\sqrt{C}\eta)(\eta^2 + 1)\\
    \zeta' = (\eta^2 + 1)(2\zeta^2 \pm\sqrt{C}\eta\zeta -\eta)
    \end{cases}
\end{equation}
The point $\eta=\zeta=0$ is the only equilibrium of the regularised system (and hence it is isolated); however, it can be easily checked to be a  degenerate one. In order to determine the behaviour of the system on the trajectories nearby, we blow up the point $(0,0)$ in $(\eta,\zeta)$-plane (we will reprise this technique later on for the more general case). The change of variables (non-homogenous blow-up as in  \cite{dumortier2006qualitative}) $\eta =r^2\sin(\phi), \ \zeta = r\cos(\phi)$ gives the system 
\[
\begin{cases}
\dot{r}= \frac{r^2\left((\sin(\phi)-2) \cos(\phi)\pm r \sqrt{C} \sin(\phi)\right)  \left(r^4\sin^2(\phi)+1\right)}{\cos^2(\phi)-2}\\
\dot{\phi}= -\frac{r \sin(\phi) \left(r^4\sin^2(\phi)+1\right) \left(\pm\sqrt{C}\sin(\phi) \cos(\phi) r-2 \cos^2(\phi)+2 \sin(\phi)\right)}{\cos^2(\phi)-2}
\end{cases}
\
\]
which has a fixed equilibrium when $r=0$. Dividing it by $r$ yields
\begin{equation}
\label{eq: sys invariant plane divided}
\begin{cases}
\dot{r}=  \frac{r\left((\sin(\phi) - 2)\cos(\phi) \pm r\sqrt{C}\sin(\phi)\right)\left(1 + r^4\sin^2(\phi)\right)}{\cos^2(\phi) - 2 }\\
\dot{\phi} = -\frac{\sin(\phi)\left(r^4\sin^2(\phi)+1\right)\left(\pm\sin(\phi)\cos(\phi)\sqrt{C}r - 2\cos^2(\phi) + 2\sin(\phi)\right)}{\cos^2(\phi) - 2}
\end{cases}
\end{equation}

When $r=0$, $\dot{r} = 0$, and  the equation $\dot{\phi} =0$ reduces to 
\[
-\frac{2\sin(\phi)( \sin(\phi)-\cos^2(\phi))}{\cos^2(\phi) - 2}=0,
\]
with solutions
\begin{equation}
    \label{eq:solutions in phi}
    \begin{sqcases}
       \phi = 0,\pi;\\
       \phi = \arctan\left(\sqrt{\frac{1}{2}\left(\sqrt{5}-1\right)}\right),\pi- \arctan\left(\sqrt{\frac{1}{2}\left(\sqrt{5}-1\right)}\right)
    \end{sqcases}
\end{equation}
For $\phi = 0,\pi$, the matrix of the linearised system is, respectively,
\[
\begin{pmatrix}
2&0\\
0&-2
\end{pmatrix} \ \mathrm{and }\ \begin{pmatrix}
-2&0\\
0&2
\end{pmatrix},
\]
signifying that $(r,\phi) = (0,0)$ and $(r,\phi) = (0,\pi)$ are both saddle points. 

When $\phi$ is the third solution from (\ref{eq:solutions in phi}), the Jacobian is 
\[
\begin{pmatrix}
\sqrt{\frac{ \sqrt{5}-1}{2}}& 0\\
\frac{\pm8\sqrt{C} - 4}{\sqrt{2 + 2\sqrt{5}}\left(5 + \sqrt{5}\right)}& \frac{2\sqrt{2}}{\sqrt{ \sqrt{5}-1}\left(\sqrt{5} + 1\right)}
\end{pmatrix}
\]
and when it is the fourth, 
\[
\begin{pmatrix}
-\sqrt{\frac{ \sqrt{5}-1}{2}}& 0\\
-\frac{\pm8\sqrt{C} - 4}{\sqrt{2 + 2\sqrt{5}}\left(5 + \sqrt{5}\right)}& -\frac{2\sqrt{2}}{\sqrt{ \sqrt{5}-1}\left(\sqrt{5} + 1\right)}.
\end{pmatrix}
\]
Hence, $(r,\phi) = \left(0,\arctan\left(\frac{\sqrt{5} - 1}{\sqrt{ 2\sqrt{5}-2}}\right)\right) $ is a repelling node and  $(r,\phi) = \left(0,-\arctan\left(\frac{\sqrt{5} - 1}{\sqrt{ 2\sqrt{5}-2}}\right)\right) + \pi $ is an attracting one. 

The dynamics on the blow-up circle are depicted in Figure \ref{fig: blow-up in the invatiant plane}, and the phase portrait in Figure \ref{fig:my_label}.

\begin{figure}
    \centering
    \includegraphics[scale=.4]{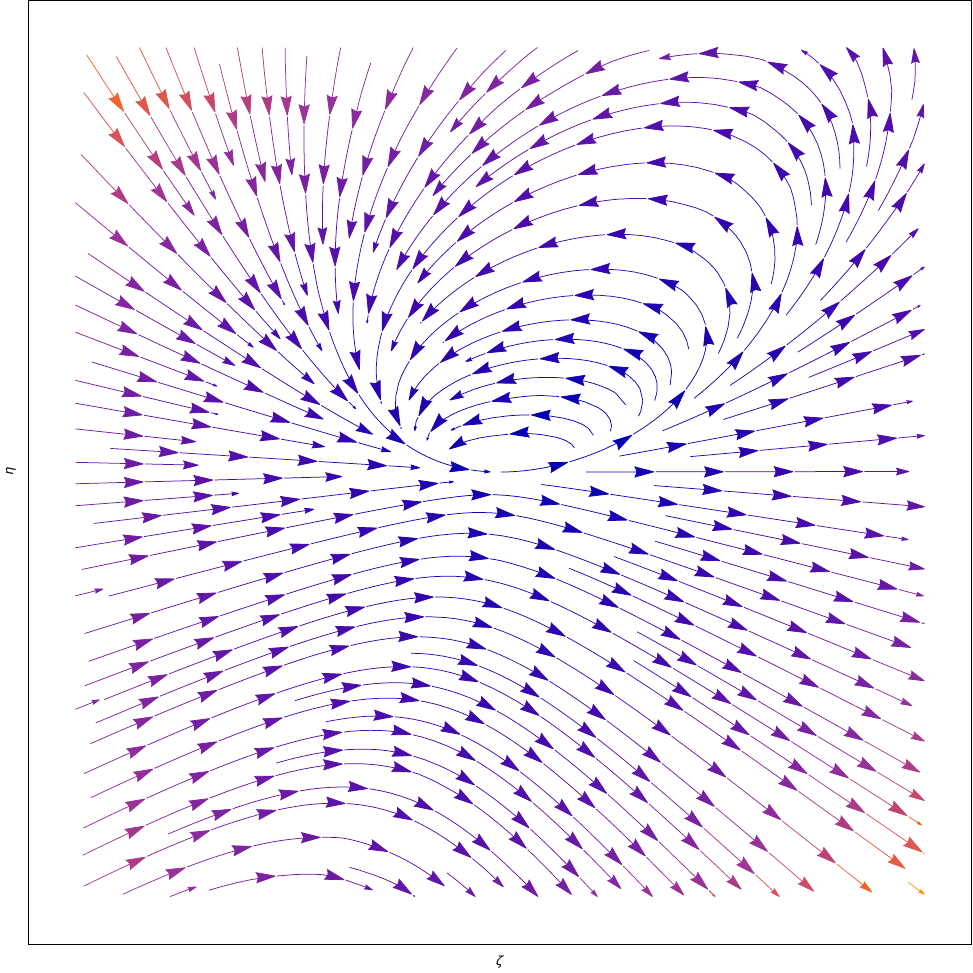}
    
    \caption{Local dynamics near the point $\zeta= \eta=0$. The plot legend describes the magnitude of the vector field at the point.}
    \label{fig:my_label}
\end{figure}

\begin{rmk}
All of the circle-like trajectories in the left part of Figure \ref{fig:my_label} are closed trajectories, passing through the origin. The magnitude of the vector field near it is almost zero; it takes infinite time to reach the origin.
\end{rmk}
\subsection{General case}

In this section, we aim to regularise collision orbits in the general case. In order to do so, we introduce the same change of variables as above: $\eta = \frac{1}{\xi}, \ \zeta = \frac{p}{\xi}$.The new system of equations is 
\begin{equation}
    \label{eq: regular prelim}
    \begin{cases}
    \dot{m}_1 = \frac{1}{\eta^2}\left(\left(m_3^2-m_2^2\right)\eta + 2m_2m_3\right);\\
    \dot{m}_2 = \frac{1}{\eta^2}\left(m_1m_2\eta - m_3\zeta\eta  - 2m_1m_3\right);\\
    \dot{m}_3 = \frac{1}{\eta}\left(m_2\zeta-m_1m_3\right);\\
    \dot{\eta} = \frac{\eta^2 +1}{\eta}\left(2\zeta - m_1\eta\right);\\
    \dot{\zeta} = \frac{\eta^2 +1}{\eta^2}\left(2m_3^2 - (1+m_2m_3)\eta + 2\zeta^2 - m_1\eta\zeta\right).
    \end{cases}
\end{equation}
The second step is identical, too: a change to fictitious time  $\tau$, with $\mathrm{d}t = \eta^2\mathrm{d}\tau$. The equations then take the form

\begin{equation}
    \label{eq:reduction}
    \begin{cases}
     m'_1 = \left(m_3^2-m_2^2\right)\eta + 2m_2m_3;\\
    m'_2 =m_1m_2\eta - m_3\zeta_\eta  - 2m_2m_3;\\
    m'_3 = \eta\left(m_2\zeta-m_1m_3\right);\\
    \eta' = (\eta^2 +1)\eta\left(2\zeta - m_1\eta\right);\\
    \zeta' = (\eta^2 +1)\left(2m_3^2 - (1+m_2m_3)\eta + 2\zeta^2 - m_1\eta\zeta\right),
    \end{cases}
 \end{equation}
 
which is our final regularised system.  As it was mentioned above, when the system tends to a collision, $\zeta,\ \eta, \ m_3\to 0$. This means that all points with coordinates $(m_1,m_2,0,0,0)$ are, in fact, equilibria (though, importantly, not all of them are reached by the system). Unsurprisingly, all of these equilibria turn out to be degenerate, and, once again, we have to employ the blow-up technique to investigate the dynamics nearby.

The standard strategy in this case would be using the five-dimensional blow-up together with  various charts to cover the entire sphere (a four-dimensional example of this technique, performed on some selected charts, can be found in \cite{jardon2016analysis}) ; however, since equilibria form a two-dimensional plane (as a result, there are no dynamics in said plane), we can 'get away' with blowing up in three dimensions: namely, $m_3,\ \eta,$ and $ \zeta$. The linearised matrices of the system will have two zero eigenvalues, but that is expected from the nature of our blow-up. 

We make the coordinate change $m_3 = r\cos(q_1),\  \eta = r^2\sin(q_1)\cos(q_2),\  \zeta = r\sin(q_1)\sin(q_2) $, leaving $m_1$ and $m_2$ intact, leading to a non-homogeneous blow-up, where $(q_1, q_2)$ are angular coordinates with $q_1\in [0, \pi]$ and $q_2\in [0,2)$.

 The vector field, as written in the new coordinates, has $r=0$ as an equilibrium, so division by $r$ is required. To save space, we  put  the formulae for the  new vector field in  Appendix \ref{sec: appendix} - the equation in question is (\ref{eq: initial vf regularised}). After division by $r$ and setting $r=0$, (\ref{eq: initial vf regularised}) becomes
\begin{equation}
\label{eq: r=0 regularisation}
    \begin{split}
        &2\cos(q_1) m_2\frac{\partial}{\partial m_1} - 2\cos(q_1) m_1 \frac{\partial}{\partial m_2}  - \frac{\sin(q_2)\cos(q_1) ( \sin(q_1)\cos(q_2) - 2) }{1 +\sin^2(q_1)\cos^2(q_2)}\frac{\partial}{\partial q_1} -\\&- \frac{\cos(q_2) \left( 2\sin^4(q_1)\cos^2(q_2)-2 + \left(2- \cos^2(q_1)   \right)\sin(q_1)\cos(q_2)\right)}{\left((1 +\sin^2(q_1)\cos^2(q_2)\right)  \sin(q_1)}\frac{\partial}{\partial q_2}
    \end{split}
\end{equation}
\begin{rmk}
Observe that $q_1$ and $q_2$-components of the vector field do not depend on $m_1$ and $m_2$ - this allows us to draw the dynamics on the blow-up sphere, unconcerned with the behaviour of  $m_1$ and $m_2$. 
\end{rmk}
As it is apparent from \eqref{eq: r=0 regularisation}, all equilibria must have $\cos(q_1)=0$ and, since due to the nature of our coordinate change  $q_1\in[0, \pi]$, $q_1 = \frac{\pi}{2}$. This  turns the coefficient at $\frac{\partial}{\partial q_1}$ to 0, and we only need to deal with the last component, which gives the equation
\[
\frac{-2\cos(q_2)\left( \cos(q_2) - \sin^2(q_2)\right)}{\cos^2(q2) + 1} =0
\]
with solutions $q_2 = \arccos\left(\frac{\sqrt{5} - 1}{2}\right), 2\pi- \arccos\left(\frac{\sqrt{5} - 1}{2}\right)$ and $q_2 = \frac{\pi}{2}, \frac{3\pi}{2}$.

Since $q_2\in[0,2\pi)$, these are all possible equilibria. Before determining their type, recall that in the new vector field $m_1$ and $m_2$ have nontrivial dynamics; we describe them first. 

When we are near the surface of the blow-up sphere,  $m_3$ is almost zero; therefore, the leading components of  $m_1$ and $m_2$ belong to a circle  $m_1^2+m_2^2=C$.  From (\ref{eq: r=0 regularisation}), $m_1' = 2\cos(q_1)m_2 + \overline{O}(r)$ and $m_2' = -2\cos(q_1)m_1 + \overline{O}(r)$, meaning that as we approach an equilibrium point, the behaviour of these variables can be approximated by rotation along a circle with a decreasing speed (this agrees with out predictions from the previous section); Theorem \ref{st: theorem omega limit set} makes sure that they tend to some point $(m_1^{\ast}, m_2^{\ast})$ as $\tau\to+\infty$. 

Now we direct our attention towards the dynamics on the blow-up sphere. 

The linearised matrix of the system at $(q_1, q_2) =\left(\frac{\pi}{2},\arccos\left(\frac{\sqrt{5} - 1}{2}\right) \right)$ is 
\[
\begin{pmatrix}
0& 0& \frac{(1-\sqrt{5} )m_2^2}{2}& -2m_2& 0\\
0& 0& \frac{(\sqrt{5}- 1)m_1m_2}{2}& 2m_1& 0\\
0& 0& \sqrt{\frac{\sqrt{5}-1}{2}}& 0& 0\\
0& 0& -\frac{m_2(\sqrt{5} - 1)^{\frac32}}{2\sqrt{2}}& -\sqrt{\frac{\sqrt{5}-1}{2}}&0\\
0&0&- \frac{\left(\sqrt{5} - 1\right)^{\frac32}m_1}{2\sqrt{10}}&0& \frac{-2\sqrt{2}(\sqrt{5} - 3)}{(\sqrt{5} - 1)^{\frac32}}
\end{pmatrix}.
\]
As remarked above, this matrix will have two zero eigenvalues, along with three non-zero ones. Two of the non-zero eigenvalues give information about the linearised dynamics on the sphere at the given equilibrium point, while the other eigenvalue gives information about the linearised dynamics transversal to the sphere at the equilibrium point.  For this matrix, these eigenvalues are  $-\sqrt{\frac{\sqrt{5}-1}{2}}$, $\sqrt{2\sqrt{5}-2}$ and $\sqrt{\frac{\sqrt{5}-1}{2}}$  corresponding  respectively to the vectors 
 \begin{equation*}
 \label{eq: eigenvectors}
\begin{pmatrix}
\frac{2\sqrt{2}m_2}{\sqrt{\sqrt{5}-1}}\\ -\frac{2\sqrt{2}m_1}{\sqrt{\sqrt{5}-1}}\\ 0\\ 1\\0
\end{pmatrix} \ , \ 
\begin{pmatrix}
0\\ 0\\ 0\\ 0\\ 1
\end{pmatrix} \ \mathrm{and}\  \begin{pmatrix}
0\\ 0\\ \frac{2\sqrt{5}}{\sqrt{5}-1}\\ -\frac{\sqrt{5}m_2}{2}\\m_1
\end{pmatrix} .
\end{equation*}
Observe that the first  eigenvector has nonzero $m_1, m_2$ -components, representing the nontrivial dynamics of these variables that we discussed above. With one positive, one negative eigenvalue in $q_1,q_2$,  the point $\left(\frac{\pi}{2},\arccos\left(\frac{\sqrt{5} - 1}{2}\right) \right)$ is a saddle.

Analogously, $\left(\frac{\pi}{2},2\pi-\arccos\left(\frac{\sqrt{5} - 1}{2}\right) \right)$ can be demonstrated to be a saddle point as well, with the linearised matrix $$
\begin{pmatrix}
0& 0& -\frac{m_2^2(\sqrt{5} - 1)}{2}& -2m_2& 0\\0& 0& \frac{m_1m_2(\sqrt{5} - 1)}{2}& 2m_1& 0\\0& 0& -\sqrt{\frac{\sqrt{5}-1}{2}}& 0& 0\\0& 0& \frac{m_2(\sqrt{5} - 1)^{\frac32}}{2\sqrt{2}}& \sqrt{\frac{\sqrt{5}-1}{2}}& 0\\0& 0&\frac{(\sqrt{5} - 1)^{\frac32}m_1}{2\sqrt{10}}&0&\frac{2\sqrt{2}(\sqrt{5} - 3)}{(\sqrt{5} - 1)^{\frac32}}
\end{pmatrix} 
$$
and eigenvalues $\sqrt{\frac{\sqrt{5}-1}{2}}$, $-\sqrt{2\sqrt{5}-2}$, $-\sqrt{\frac{\sqrt{5}-1}{2}}$ corresponding to 
\begin{equation*}
 \label{eq: eigenvectors2}
\begin{pmatrix}
-\frac{2\sqrt{2}m_2}{\sqrt{\sqrt{5}-1}}\\ \frac{2\sqrt{2}m_1}{\sqrt{\sqrt{5}-1}}\\ 0\\ 1\\0
\end{pmatrix} \ ,\ 
\begin{pmatrix}
0\\ 0\\ 0\\ 0\\ 1
\end{pmatrix} \ \mathrm{and} \ \begin{pmatrix}
0\\ 0\\ \
\frac{2\sqrt{5}}{\sqrt{5}-1}\\ -\frac{\sqrt{5}m_2}{2}\\ m_1
\end{pmatrix}
\end{equation*}

At the point $\left(\frac{\pi}{2},\frac{\pi}{2}\right)$  the matrix is
$$\begin{pmatrix}
0& 0& 0& -2m_2& 0\\0& 0& 0& 2m_1& 0\\0& 0& 2& 0&0\\0&0& 0&-2& 0\\0& 0& 0& 0& -2
\end{pmatrix}
$$
with eigenvalues $-2,  -2$ and their eigenvectors
$$
\begin{pmatrix}
0\\0\\0\\0\\1
\end{pmatrix} \ \mathrm{and} \ 
\begin{pmatrix}
m_2\\ -m_1\\0\\0\\1\\0
\end{pmatrix},
$$
making it an attracting node.

Lastly, $\left(\frac{\pi}{2},\frac{3\pi}{2}\right)$ has
$$
\begin{pmatrix}
0& 0& 0 &-2m_2& 0\\0& 0& 0 &2m_1 &0\\0 &0& -2 &0& 0\\0& 0& 0& 2& 0\\0& 0& 0& 0& 2
\end{pmatrix}
$$
with eigenvalues $2,\ 2$ and  eigenvectors
$$
\begin{pmatrix}
0\\0\\0\\0\\1
\end{pmatrix},   \ \
\begin{pmatrix}
-m_2\\ m_1\\0\\1\\0
\end{pmatrix}
$$
and is a repelling node. 
\begin{figure}
    \centering
  \subfigure[Parametrisation in $q_1, q_2$]{\includegraphics[scale=.45]{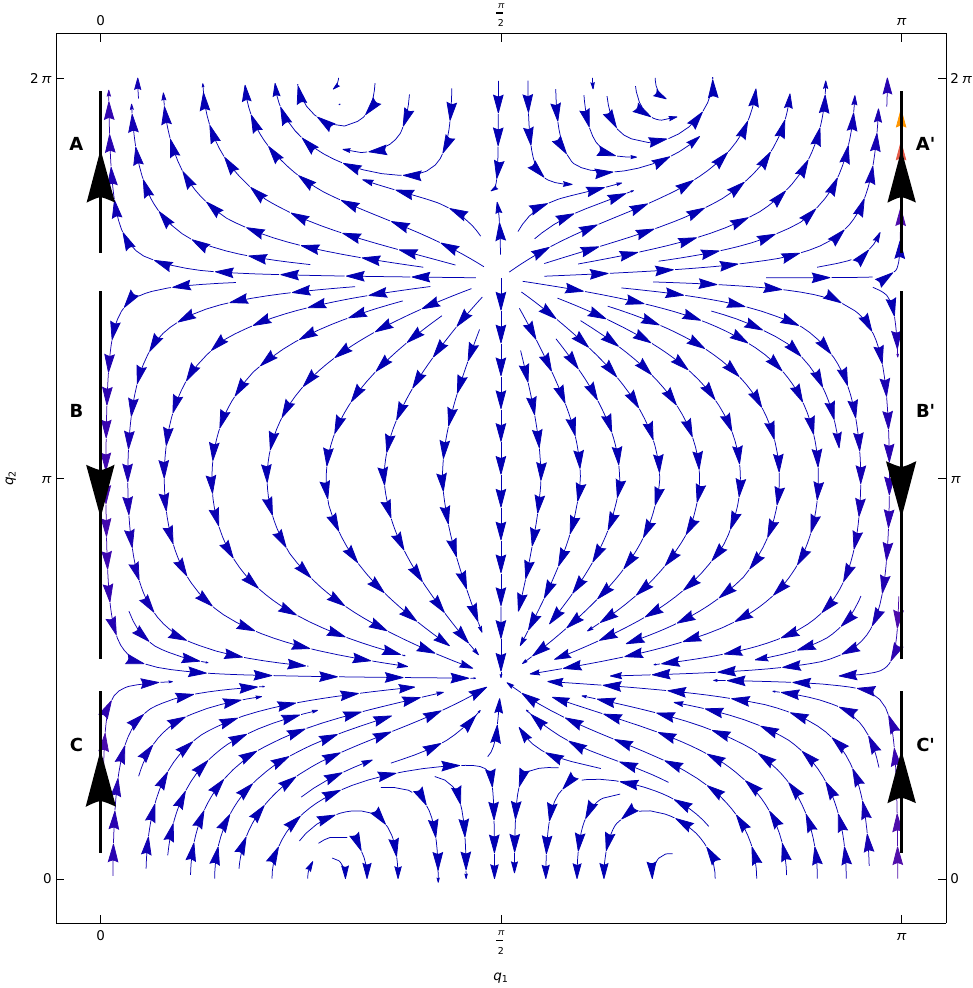}}
  \subfigure[Parametrisation in $Q_1, Q_2$]{\includegraphics[scale=.45]{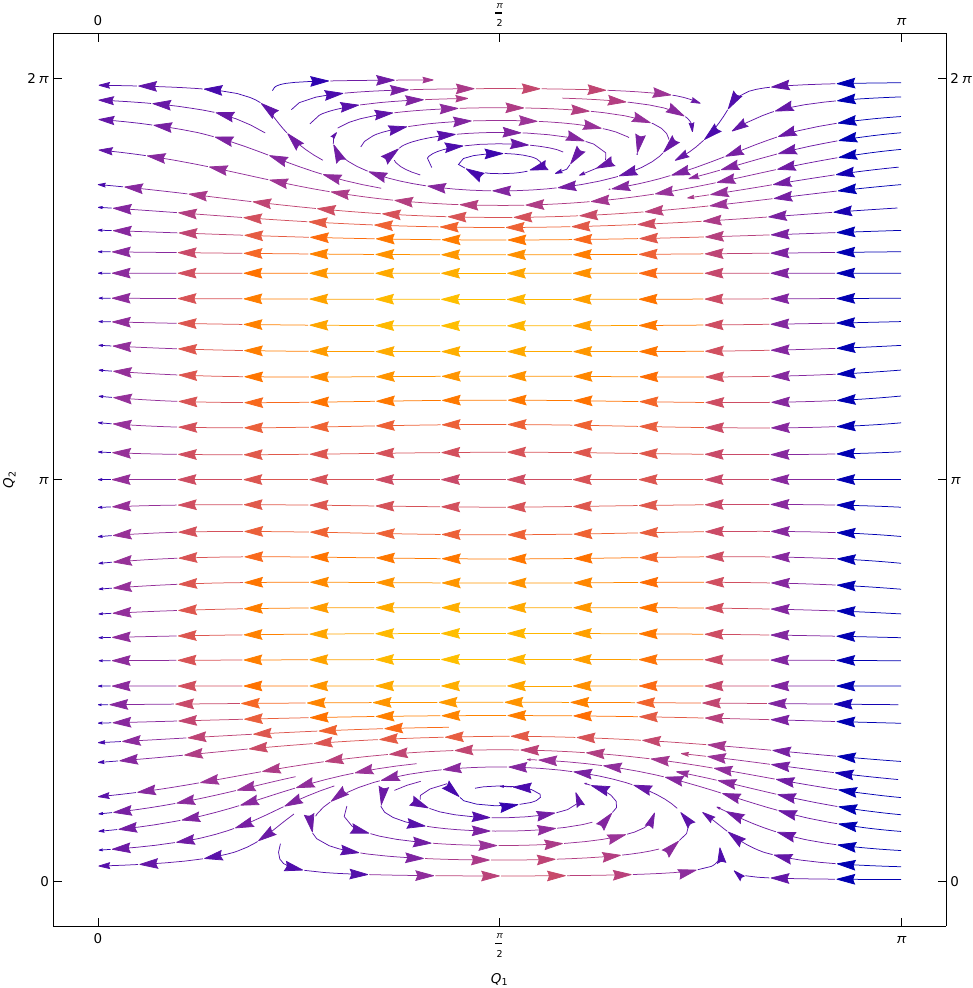}}
    \caption{Two charts on the blow-up sphere}
    \label{fig: regularisation}
\end{figure}

Therefore, so far we have discovered two saddles, a repelling and an attracting node for the dynamics on the sphere. However, in accordance with Poincare-Hopf theorem (\cite{milnor1997topology}), there must be at least two more critical points  on the sphere that are either centres or nodes;  we plot the vector field to determine: it is depicted in  Figure \ref{fig: regularisation} (a). As it can be seen, we indeed have only two additional critical points on the sphere: supposedly centres. A computation gives that they are  located at $(q_1,q_2) = \left(0, \arcsin\left(\frac{\sqrt{17}-1}{4}\right)\right)$ and $  \left(0, \pi-\arcsin\left(\frac{\sqrt{17}-1}{4}\right)\right)$. 

\begin{rmk}
We wish to stress that the last two points are equilibria in $q_1$ and $q_2$ only; they are \textbf{not} equilibria of the entire system. 
\end{rmk}

In the light of this, we need to  consider the Jacobians only with the last three components of the vector field.   The two respective linearised matrices are

$$
\begin{pmatrix}
0&0&0\\
\frac{\sqrt{34}(\sqrt{17}-1)^{\frac32}}{136}&0&\frac{\sqrt{17}-9}{\sqrt{34}\sqrt{\sqrt{17}-1}}\\\frac{\sqrt{\sqrt{17}-1}m_2}{-2\sqrt{2}}&-\frac{ \sqrt{136} (\sqrt{17}-9)}{ (\sqrt{17})^\frac{5}{2}}&0
\end{pmatrix}\ , \ \begin{pmatrix} 0&0&0\\
-\frac{\sqrt{34}(\sqrt{17}-1)^{\frac32}}{136}&0&\frac{9-\sqrt{17}}{\sqrt{34}\sqrt{\sqrt{17}-1}}\\ -\frac{\sqrt{\sqrt{17}-1}m_2}{2\sqrt{2}}&\frac{ \sqrt{136} (\sqrt{17}-9)}{ (\sqrt{17})^\frac{5}{2}}&0
\end{pmatrix}$$

with the same pair of imaginary eigenvalues 
$\sqrt{\frac{1-\sqrt{17}}{2}},\ -\sqrt{\frac{1-\sqrt{17}}{2}}$ in $q_1$, $q_2$ , definitively making them centres. As it can be easily seen, the eigenvalues in $r$-direction are 0 for both matrices. 

Of course, we analyzed the dynamics using only one chart, which is very-well known not to be enough to cover $S^2$. Moreover, since the (local) diffeomorphism between the chart we are using and a (suitable) subset of $S^2$ is singular at some points, we might have introduced spurious equilibria. In order to check this, we perform our analysis using a second chart.
More specifically, when $q_1=0,\pi$, the vector field (\ref{eq: r=0 regularisation}) is not defined. In $m_3,\ \eta$ and $\zeta$ these values correspond to the points of type $(\pm r,0,0)$; that entails that in order to finish our investigation, we need a second chart that does not have singularities at these two points on the blow-up sphere.
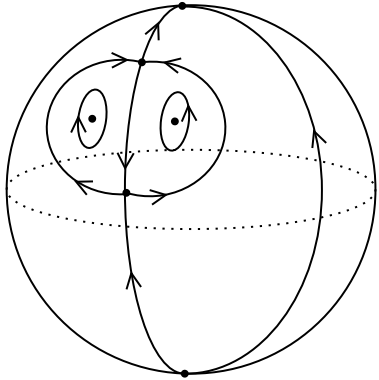
\begin{figure}
    \centering

\tikzset{every picture/.style={line width=0.75pt}} 

\begin{tikzpicture}[x=0.75pt,y=0.75pt,yscale=-1,xscale=1]

\draw   (192,99) .. controls (192,47.64) and (233.64,6) .. (285,6) .. controls (336.36,6) and (378,47.64) .. (378,99) .. controls (378,150.36) and (336.36,192) .. (285,192) .. controls (233.64,192) and (192,150.36) .. (192,99) -- cycle ;
\draw  [dash pattern={on 0.84pt off 2.51pt}] (192,99) .. controls (192,87.95) and (233.64,79) .. (285,79) .. controls (336.36,79) and (378,87.95) .. (378,99) .. controls (378,110.05) and (336.36,119) .. (285,119) .. controls (233.64,119) and (192,110.05) .. (192,99) -- cycle ;
\draw  [draw opacity=0] (281.82,192) .. controls (281.82,192) and (281.82,192) .. (281.82,192) .. controls (265.15,192) and (251.64,150.36) .. (251.64,99) .. controls (251.64,47.64) and (265.15,6) .. (281.82,6) .. controls (281.88,6) and (281.94,6) .. (282,6) -- (281.82,99) -- cycle ; \draw   (281.82,192) .. controls (281.82,192) and (281.82,192) .. (281.82,192) .. controls (265.15,192) and (251.64,150.36) .. (251.64,99) .. controls (251.64,47.64) and (265.15,6) .. (281.82,6) .. controls (281.88,6) and (281.94,6) .. (282,6) ;  
\draw   (261.87,20.69) -- (268.4,15.23) -- (268.66,23.74) ;
\draw   (233.21,78.04) .. controls (229.38,77.53) and (227.15,70.53) .. (228.23,62.41) .. controls (229.31,54.29) and (233.29,48.12) .. (237.12,48.63) .. controls (240.95,49.14) and (243.18,56.13) .. (242.11,64.26) .. controls (241.03,72.38) and (237.04,78.55) .. (233.21,78.04) -- cycle ;
\draw   (274.88,79.37) .. controls (271.05,78.86) and (268.82,71.87) .. (269.89,63.74) .. controls (270.97,55.62) and (274.96,49.45) .. (278.79,49.96) .. controls (282.62,50.47) and (284.85,57.47) .. (283.77,65.59) .. controls (282.69,73.71) and (278.71,79.88) .. (274.88,79.37) -- cycle ;
\draw  [draw opacity=0] (252.03,101.39) .. controls (251.72,101.41) and (251.42,101.43) .. (251.11,101.44) .. controls (230.4,102.27) and (213,87.75) .. (212.26,69.02) .. controls (211.51,50.3) and (227.7,34.45) .. (248.41,33.62) .. controls (252.5,33.46) and (256.47,33.89) .. (260.2,34.84) -- (249.76,67.53) -- cycle ; \draw   (252.03,101.39) .. controls (251.72,101.41) and (251.42,101.43) .. (251.11,101.44) .. controls (230.4,102.27) and (213,87.75) .. (212.26,69.02) .. controls (211.51,50.3) and (227.7,34.45) .. (248.41,33.62) .. controls (252.5,33.46) and (256.47,33.89) .. (260.2,34.84) ;  
\draw  [draw opacity=0] (260.2,34.84) .. controls (260.81,34.77) and (261.43,34.7) .. (262.05,34.65) .. controls (282.71,32.98) and (300.69,46.77) .. (302.2,65.45) .. controls (303.71,84.13) and (288.19,100.63) .. (267.53,102.31) .. controls (262.22,102.74) and (257.09,102.15) .. (252.37,100.71) -- (264.79,68.48) -- cycle ; \draw   (260.2,34.84) .. controls (260.81,34.77) and (261.43,34.7) .. (262.05,34.65) .. controls (282.71,32.98) and (300.69,46.77) .. (302.2,65.45) .. controls (303.71,84.13) and (288.19,100.63) .. (267.53,102.31) .. controls (262.22,102.74) and (257.09,102.15) .. (252.37,100.71) ;  
\draw  [fill={rgb, 255:red, 0; green, 0; blue, 0 }  ,fill opacity=1 ] (252.01,99.27) .. controls (251.22,99.47) and (250.73,100.28) .. (250.93,101.07) .. controls (251.13,101.87) and (251.93,102.35) .. (252.73,102.15) .. controls (253.52,101.95) and (254,101.15) .. (253.81,100.35) .. controls (253.61,99.56) and (252.8,99.08) .. (252.01,99.27) -- cycle ;
\draw  [fill={rgb, 255:red, 0; green, 0; blue, 0 }  ,fill opacity=1 ] (259.84,33.41) .. controls (259.05,33.61) and (258.57,34.41) .. (258.77,35.2) .. controls (258.96,36) and (259.77,36.48) .. (260.56,36.28) .. controls (261.36,36.08) and (261.84,35.28) .. (261.64,34.49) .. controls (261.44,33.69) and (260.64,33.21) .. (259.84,33.41) -- cycle ;
\draw  [fill={rgb, 255:red, 0; green, 0; blue, 0 }  ,fill opacity=1 ] (280.2,4.92) .. controls (279.41,5.12) and (278.92,5.93) .. (279.12,6.72) .. controls (279.32,7.51) and (280.13,8) .. (280.92,7.8) .. controls (281.71,7.6) and (282.2,6.8) .. (282,6) .. controls (281.8,5.21) and (280.99,4.72) .. (280.2,4.92) -- cycle ;
\draw  [fill={rgb, 255:red, 0; green, 0; blue, 0 }  ,fill opacity=1 ] (281.46,190.56) .. controls (280.67,190.76) and (280.18,191.57) .. (280.38,192.36) .. controls (280.58,193.15) and (281.38,193.64) .. (282.18,193.44) .. controls (282.97,193.24) and (283.46,192.43) .. (283.26,191.64) .. controls (283.06,190.85) and (282.25,190.36) .. (281.46,190.56) -- cycle ;
\draw  [fill={rgb, 255:red, 0; green, 0; blue, 0 }  ,fill opacity=1 ] (234.81,61.9) .. controls (234.01,62.09) and (233.53,62.9) .. (233.73,63.69) .. controls (233.93,64.49) and (234.73,64.97) .. (235.53,64.77) .. controls (236.32,64.57) and (236.8,63.77) .. (236.6,62.97) .. controls (236.41,62.18) and (235.6,61.7) .. (234.81,61.9) -- cycle ;
\draw  [fill={rgb, 255:red, 0; green, 0; blue, 0 }  ,fill opacity=1 ] (276.47,63.23) .. controls (275.68,63.43) and (275.2,64.23) .. (275.4,65.03) .. controls (275.59,65.82) and (276.4,66.3) .. (277.19,66.1) .. controls (277.99,65.91) and (278.47,65.1) .. (278.27,64.31) .. controls (278.07,63.51) and (277.27,63.03) .. (276.47,63.23) -- cycle ;
\draw   (256.1,80.74) -- (251.76,88.28) -- (247.89,80.49) ;
\draw   (252.45,149.44) -- (254.88,141.27) -- (259.79,148.24) ;
\draw   (278.46,40.17) -- (271.75,34.91) -- (280.03,32.89) ;
\draw   (244.92,30.09) -- (252.67,33.64) -- (245.08,37.53) ;
\draw   (232.41,101.71) -- (227.02,95.11) -- (235.55,94.96) ;
\draw   (264.12,99.33) -- (272.28,101.77) -- (265.32,106.68) ;
\draw   (280.33,64.76) -- (283.62,56.89) -- (287.76,64.34) ;
\draw   (224.5,70.28) -- (228.11,62.56) -- (231.94,70.17) ;
\draw   (346.16,78.1) -- (347.02,69.62) -- (353.15,75.54) ;
\draw  [draw opacity=0] (282.06,6.72) .. controls (282.06,6.72) and (282.06,6.72) .. (282.06,6.72) .. controls (320.14,6.72) and (351,48.2) .. (351,99.36) .. controls (351,150.52) and (320.14,192) .. (282.06,192) -- (282.06,99.36) -- cycle ; \draw   (282.06,6.72) .. controls (282.06,6.72) and (282.06,6.72) .. (282.06,6.72) .. controls (320.14,6.72) and (351,48.2) .. (351,99.36) .. controls (351,150.52) and (320.14,192) .. (282.06,192) ;

\end{tikzpicture}

    \caption{Schematic representation of the  vector field on the blow-up sphere}
    \label{fig: sphere vf}
\end{figure}

The alternative parametrisation we choose is $m_3 = R\sin\left(Q_1\right)\sin\left(Q_2\right)$, $\eta = R^2\sin\left(Q_1\right)\cos\left(Q_2\right) $, $\zeta = R\cos\left(Q_1\right)$. Now, the transformation is singular at the points $(0,0,R)$; but importantly, $(\pm r,0,0)$ correspond to $Q_1= \frac{\pi}{2}$, $Q_2=\frac{\pi}{2}$ or $Q_2=\frac{3\pi}{2}$, points strictly inside the chart.  

Identically to the previous case, we rewrite the vector field in the new variables (the original equation in new variables is (\ref{eq: the other long}) in Appendix \ref{sec: appendix}) and after dividing by $R$ and subsequently substituting $R=0$, we obtain
\begin{equation*}
    \begin{split}
&2m_2\sin(q_2)\sin(q_1)\frac{\partial}{\partial m_1}  -2\sin(q_1)\sin(q_2)m_1\frac{\partial}{\partial m_2}  +\\&+ \frac{\sin(q_1)\left(2\sin^2(q_1)\cos(q_2)^4 +\cos^3(q_2)\sin(q_1) + \cos(q_2)\sin(q_1) - 2\right)}{1 + \sin(q_1)^2\cos^2(q_2)}\frac{\partial}{\partial Q_1} +\\&+ \frac{\cos^2(q_2)\cos(q_1)\sin(q_2)(2\cos(q_2)\cos^2(q_1) - 2\cos(q_2) - \sin(q_1))}{1 + \sin(q_1)^2\cos^2(q_2)}\frac{\partial}{\partial Q_2}
    \end{split}
\end{equation*}
whence it is clear that $Q_1=\frac{\pi}{2}$, $Q_2 =\frac{\pi}{2},\frac{3\pi}{2}$ are not equilibrium points; therefore,  we have found all the possible equlibria on the sphere.

\begin{rmk}
It can be easily checked that the second vector field has exactly the same number of zeros of the same type on the sphere: two centres (that are equilibria on the sphere only), two saddles, and  two nodes: stable and unstable.

Moreover, for all these points their respective $m_3, \ \eta,\ \zeta$- coordinates coincide. As an example, we take a saddle  at $q_1 = \frac{\pi}{2},\  q_2 = \arccos\left(\frac{\sqrt{5}- 1}{2}\right)$. With varying $r$, this point corresponds to the points of the type  $\left(0, r^2\frac{\sqrt{5} - 1}{2}, r\sqrt{\frac{\sqrt{5}-1}{2}}\right)$ in $m_3, \ \eta,\ \zeta$; so does the saddle point at $Q_2 = 0,\  Q_1 = \arcsin\left(\frac{\sqrt{5}- 1}{2}\right)$.

\end{rmk}
The vector field on the second chart is given in Figure \ref{fig: regularisation}(b).

The last question that we have to answer is that of reconstructing the vector field on the sphere from our charts. 

It is a significantly easier task for the second parametrisation, and so we do it first. Since $Q_1=0,\pi$ correspond to one point each, we 'shrink' the two intervals to a point, and then identify the two lines $Q_2=0, 2\pi$. In the end we get a sphere with nodes at  North and South poles. 

Gluing the sphere with the first parametrisation is a bit more fiddly: we have to stretch the line $q_1=\frac{\pi}{2}$, so that the nodes lie at the poles. Then we twist the lines marked by $A, \ A',\  B,\ B',\ C, \ C'$ so that they are vertical and glue them together, as well as the edges $Q_2=0$ and $Q_2 = 2\pi$. Now all of these lines are trajectories connecting the North and the South pole, corresponding to the horizontal lines in the picture on the right.

The schematic form of the vector field on the sphere is depicted in Figure \ref{fig: sphere vf}.

\section{Conclusion}
In this paper we have considered collision trajectories for the two-body problem on a sphere and investigated the near-collision trajectories via blow-up. 

In order to simplify our computations, we have chosen the case of two identical bodies; we believe that our proofs can be adapted to the general case -- comparison of collision trajectories with this symmetric case could be an interesting topic. 

Additionally, similar investigation can be carried out for the case of negative curvature. In the light of this, another possible topic of investigation would be the families of collision orbits on curved surfaces and the way they change with the curvature $\kappa$, as it comes from, say, negative to zero to positive. 
\appendix
\section{Appendix}

The vector field, as rewritten in the new variables $r,\ q_1,\ q_2$ has the form
\label{sec: appendix}

\begin{equation}
\label{eq: initial vf regularised}
 \begin{split}
     &\biggl(\bigl(r^3\cos^2(q_1) + m_2^2 r\bigr)\cos(q_2)  \sin(q_1) +2\cos(q_1) m_2\biggr) r\frac{\partial}{\partial m_1}+ \\&+ r \biggl(\sin(q_2)\cos^3(q_1)\cos(q_2) r^3 - \bigl(r^3\cos(q_2)\sin(q_2) + 2 m_1\bigr)\cos(q_1) +  \sin(q_1)\cos(q_2) m_1 m_2 r\biggr)\frac{\partial}{\partial m_2}+\\&+ f_3\frac{\partial}{\partial r}+f_4\frac{\partial}{\partial q_1} + f_5\frac{\partial}{\partial q_2},
    \end{split}
\end{equation}

where

\begin{footnotesize}
\begin{equation*}
\begin{split}
   f_3 =  \frac{1}{(-1 + (\cos^2(q_1) - 1)\cos^2(q_2)} \biggl(&r^2  \sin(q_1) \bigl( \sin^3(q_1)r^4 \cos^3(q_2) \left((m_2 r \cos(q_1) +  1)\sin(q_2) + r m_1\sin(q_1)\right) +\\&+ 4r^8\sin^4(q_1) \sin(q_2)\cos^2(q_2) + (m_1 r +  \sin(q_1)\sin(q_2))\cos(q_2) - 2\sin(q_2)\bigr)\biggr),
   \end{split}
\end{equation*}
\end{footnotesize}
\begin{footnotesize}
\begin{equation*}
\begin{split}
   f_4=- \frac{r}{1 + \sin^2(q_1)\cos^2(q_2)} \biggl(&\bigl( m_2 r^4  \sin(q_1)\sin(q_2)\cos^2(q_1) -m_1 r^4\cos^3(q_1) + (m_1(r^4-1) + r^3  \sin(q_1)\sin(q_2) )\cos(q_1)+\\& + m_2\sin(q_2)  \sin(q_1)\bigr)  r \sin^2(q_1)\cos^3(q_2) - 2 r^4\cos(q_1)\sin^2(q_1)\sin(q_2)\cos^2(q_2) + \\&+ \sin(q_1)\sin(q_2) (m_2 r + \cos(q_1))\cos(q_2) - 2\sin(q_2)\cos(q_1)\biggr) 
    \end{split}
   \end{equation*}
   \end{footnotesize}
   and
   
   \begin{footnotesize}
   
    \begin{equation*}
    \begin{split}
    f_5 = \frac{\cos(q_2) r}{1 + \sin^2(q_1)\cos^2(q_2)} \biggl(&2 - 2 r^4 \sin^6(q_1)\cos(q_2)^4 + \bigl(m_1 r^4\cos^4(q_1)\sin(q_2) - m_2 r^4  \sin(q_1)\cos^3(q_1) + \\&+(-2 m_1 r^4\sin(q_2) - r^3  \sin(q_1))\cos^2(q_1) + (2 r^4 + 1)  \sin(q_1) m_2\cos(q_1) +\\&+ r^3 (m_1 r\sin(q_2) + 2  \sin(q_1))\bigr) (\cos(q_1) + 1) r (\cos(q_1) - 1)\cos^3(q_2) +\\&+ (-2\cos^4(q_1) + (-2 r^4 + 4)\cos^2(q_1) + 2 r^4 - 2)\cos^2(q_2) + \bigl((m_1 r\sin(q_2) +  \sin(q_1))\cos^2(q_1) -\\&- m_2 r  \sin(q_1)\cos(q_1) - m_1 r\sin(q_2) - 2  \sin(q_1)\bigr)\cos(q_2)\biggr)
    \end{split}
    \end{equation*}
\end{footnotesize}
The regularised vector field, as rewritten in the variables $R,\ Q_1,\ Q_2$  has the form 
\begin{equation}
    \label{eq: the other long}
    \begin{split}
    &\sin(Q_1)R\biggl(R^3\sin^2(Q_1)\cos^3(Q_2) - R(R^2\cos^2(Q_1) + m_2^2 - R^2)\cos(Q_2) + 2m_2\sin(Q_2)\biggr)\frac{\partial}{\partial m_1}-\\&- R\biggl(2\sin(Q_2)m_1+R(R^2\sin(Q_1)\sin(Q_2)\cos(Q_1) - m_1m_2)\cos(Q_2)  \biggr)\sin(q_1)\frac{\partial}{\partial m_2}+\\&+ g_3\frac{\partial}{\partial R}+ g_4\frac{\partial}{\partial Q_1}+ g_5\frac{\partial}{\partial Q_2},
    \end{split}
\end{equation}
where
\begin{footnotesize}
\begin{equation*}
\begin{split}
  g_3= \frac{- R^2}{1 + \sin^2(Q_1)\cos^2(Q_2)}\biggl(&m_1R^5\sin(Q_1)\sin^4(Q_1)\cos(Q_2)^5 + R^4\cos(Q_1)\sin^2(Q_1)\bigl(-m_2R\cos^2(Q_1)\sin(Q_2)+\\&+ + m_1r\cos(Q_1)\sin(Q_1) + m_2R\sin(Q_2) + \sin(Q_1)\bigr)\cos^3(Q_2) + (2R^4\cos^3(Q_1)-\\& - 2R^4\cos(Q_1))\cos^2(Q_2) + \sin(Q_1)(m_1R + \cos(Q_1))\cos(Q_2) - 2\cos(Q_1)\biggr),
  \end{split}
\end{equation*}
\end{footnotesize}
\begin{footnotesize}
\begin{equation*}
\begin{split}
  g_4= \frac{-R \sin( Q_1)}{1 + \sin^2(Q_1) \cos^2( Q_2} \biggl(&2 - 2  R^4\sin^4(Q_1)\cos( Q_2)^6 - R^4\sin^2(Q_1)\bigl(m_2  R \sin^2( Q_1) \sin( Q_2) + \sin( Q_1)\bigr) \cos( Q_2)^5 + \\&+(2 \cos^2( Q_1) - 2) \cos( Q_2)^4 + ( R^4 \cos^2( Q_1) -  R^4 - 1) (-m_2  R \cos^2( Q_1) \sin( Q_2) +\\&+ m_1  R \cos( Q_1) \sin( Q_1) + m_2  R \sin( Q_2) + \sin( Q_1)) \cos^3( Q_2) +  2  R^4\sin^2(Q_1) \cos^2( Q_2) + \\&+(-m_2  R \sin( Q_2) - \sin( Q_1)) \cos( Q_2)\biggr)  
  \end{split}
\end{equation*}
\end{footnotesize}
and
\begin{footnotesize}
\begin{equation*}
\begin{split}
  g_5=  \frac{R \cos^2( Q_2)}{1 + \sin^2( Q_1) \cos^2( Q_2)}\biggl(&m_2  R^5 \cos( Q_1) \sin^4(Q_1) \cos( Q_2)^4 - 2  R^4 \cos( Q_1) \sin( Q_2) \sin^4(Q_1) \cos^3( Q_2) +  \\&-R\sin^2(Q_1)\cos^2( Q_2) \bigl(-m_2  R^4 \cos^3( Q_1) + m_1  R^4 \cos^2( Q_1) \sin( Q_1) \sin( Q_2) +\\&+ (m_2  R^4 +  R^3 \sin( Q_1) \sin( Q_2) - m_2) \cos( Q_1) - m_1  R^4 \sin( Q_1) \sin( Q_2)\bigr)   +\\&+ (2 \cos^3( Q_1) \sin( Q_2) - 2 \cos( Q_1) \sin( Q_2)) \cos( Q_2) + (m_2  R - \sin( Q_1) \sin( Q_2)) \cos( Q_1) -\\&- \sin( Q_1) \sin( Q_2) m_1  R\biggr)  
  \end{split}
\end{equation*}
\end{footnotesize}

\end{document}